\RequirePackage{fix-cm}
\documentclass[a4paper]{svjour3arxiv}                     
\usepackage{amssymb,amsmath}
\usepackage{dsfont}
\spnewtheorem*{condition}{Condition}{\bf}{\rm}
\smartqed
\numberwithin{equation}{section}
\smartqed  
\usepackage{graphicx}
\usepackage{marvosym}
\newcommand{\envelope}{(\raisebox{-.5pt}{\scalebox{1.45}{\Letter}}\kern-1.7pt)}
\DeclareMathOperator{\sign}{sign}

\begin{document}

\title{Joint Extremal Behavior of Hidden and Observable Time Series with an Application to GARCH Processes
}
\subtitle{}

\titlerunning{Joint Extremal Behavior of Hidden and Observable Time Series}        

\author{Andree Ehlert         \and
	Ulf-Rainer Fiebig      \and
        \\  Anja Jan\ss en          \and
        Martin Schlather
}


\institute{Andree Ehlert \at
Institute of Economics, Leuphana University of L\"{u}neburg, Scharnhorststr.~1, D-21335 L\"{uneburg}, Germany.
              \email{ehlert@uni-lueneburg.de}          
              \and
            Ulf-Rainer Fiebig \at
	      Institute for Mathematical Stochastics, Georg-August-University G\"{o}ttingen, Goldschmidtstr.~7,
	      D-37077 G\"{o}ttingen, Germany.
              \email{urfiebig@math.uni-goettingen.de} 
	  \and
          Anja Jan\ss en \envelope\at
	      Department of Mathematics, SPST, University of Hamburg,
	      Bundesstr.~55,
	      D-20146 Hamburg, Germany.
              \email{anja.janssen@math.uni-hamburg.de} 
 	  \and
           Martin Schlather \at
	      School of Business Informatics and Mathematics, University of Mannheim, D-68131 Mannheim, Germany. \email{schlather@math.uni-mannheim.de} 
}

\date{Received: date / Accepted: date}

\maketitle

\begin{abstract}
We study the behavior of a real-valued and unobservable process~$(Y_t)_{t\in\mathbb Z}$ under an extreme event of a related process~$(X_t)_{t\in\mathbb Z}$ that is observable. Our analysis is motivated by the well-known GARCH model which represents two such sequences, i.e.\ the observable log returns of an asset as well as the hidden volatility process. 
Our results complement the findings of Segers [J.~Segers, Multivariate regular variation of heavy-tailed Markov chains, arXiv:math/0701411 (2007). Available online: http://arxiv.org/abs/math/0701411] and Smith [R.~L.~Smith, The extremal index for a Markov chain. J.~Appl.\ Prob.\ (1992)] for a single time series.
We show that under suitable assumptions their concept of a tail chain as a limiting process
is also applicable to our setting.
Furthermore, we discuss existence and uniqueness of a limiting process under some weaker assumptions. 
Finally, we apply our results to the GARCH$(1,1)$ case.
\keywords{ARCH processes \and GARCH processes \and extremal index \and joint extremal behavior\and  multivariate regular variation \and tail chain \and time series }
\subclass{60G70 \and 60J05}
\end{abstract}

\section{Introduction}\label{s:intro}
An extensive class of financial time series models is based on two interrelated processes. In particular, many models include an unobservable part that reflects a certain regime or the volatility of the process. A well-known example is given by the GARCH family.
It is typically applied in order to model financial log returns where the unobservable volatility process drives the observable price of an asset.
In the following, let~$(X_t)_{t\in\mathbb Z}$ denote such a process and~$(Y_t)_{t\in\mathbb Z}$ its unobservable counterpart. Let both~$(X_t)_{t\in\mathbb Z}$ and~$(Y_t)_{t\in\mathbb Z}$ be univariate.
A common approach for the analysis of the extremal behavior of such interrelated processes focusses on the joint sequence $(Z_t)_{t \in \mathbb{Z}}:=(X_t,Y_t)_{t \in \mathbb{Z}}$.
More precisely, the process is studied under the condition $\{\|Z_0\|>x\}$ for~$x \to \infty$ and an arbitrary norm~$\|\cdot\|$ on~$\mathbb{R}^2$. The connection of this approach to the concept of multivariate regular variation has been discussed extensively in~\cite{basrak}. We shall follow a more natural point of view where the process~$(Y_t)_{t \in \mathbb{Z}}$ is unobservable. That is, we analyze its limiting behavior under the (observable) event $\{|X_0|>x\}$ as $x \to \infty$. Hence, for $-\infty<m\le n <\infty$ we focus on the limit distribution of 
\begin{equation}\label{inquestion}\mathcal{L}\left(\frac{Y_m}{x}, \dots, \frac{Y_n}{x} \;\middle\vert\; |X_0|>x\right) \end{equation}
as $x\to\infty$.
We assume~$(Y_t)_{t \in \mathbb{Z}}$ to be of a simple Markovian structure, i.e.\
\begin{equation}\label{Y}Y_t=\Phi(Y_{t-1},\epsilon_t), \;\;\; t \in \mathbb{Z},\end{equation}
for some measurable mapping $\Phi:\mathbb{R}\times \mathbb{S}\to \mathbb{R}$ and some sequence~$(\epsilon_t)_{t \in \mathbb{Z}}$ of i.i.d.\ innovations on a measurable space $(\mathbb{S}, \mathcal{S})$ . Additionally, we will require the sequence of innovations~$(\epsilon_t)_{t>s}$ to be independent of~$(Y_t)_{t\leq s}$ for all~$s \in \mathbb{Z}$. 
Based on $(Y_t)_{t \in \mathbb{Z}}$ and the innovations let the observable process be given by
\begin{equation}\label{Xlonger}X_t=\Psi(Y_t,\epsilon_{t-s_-},\dots , \epsilon_{t+s_+}), \;\;\; t \in \mathbb{Z},\end{equation}
for some measurable mapping $\Psi:\mathbb{R}\times\mathbb{S}^{s_-+s_++1}\to \mathbb{R}$ with $s_+,  s_- \geq 0$. 
We will always assume that a stationary solution to~\eqref{Y} and~\eqref{Xlonger} exists.
Now, by~$\Psi$ as well as by~$s_-$ and $s_+$ we have a simple, but flexible model for the dependence between~$(X_t)_{t\in\mathbb Z}$ and~$(Y_t)_{t\in\mathbb Z}$.
However, note that from the recursive definition in~\eqref{Y} we may find a function $\widetilde{\Psi}:\mathbb{R}\times\mathbb{S}^{s_-+s_++1}\to \mathbb{R}$ such that $\widetilde{X}_t:=X_{t+s_-+1}=\widetilde{\Psi}(Y_{t}, \epsilon_{t+1}, \dots, \epsilon_{t+s})$, $t\in\mathbb Z$, with $s:=s_-+s_++1$. Hence, for ease of notation we may in the following assume that there exists an $s\geq 1$ such that
\begin{equation}\label{X}X_t=\Psi(Y_t,\epsilon_{t+1},\dots , \epsilon_{t+s}), \;\;\; t \in \mathbb{Z}.\end{equation}
We may interpret $(X_t)_{t \in \mathbb{Z}}$ and $(Y_t)_{t \in \mathbb{Z}}$ as a generalized hidden Markov model which incorporates a large class of models for financial time series, cf.\ \cite{carrasco} for the general definition.
We shall discuss the GARCH$(1,1)$ process (cf.\ \cite{bollerslev,taylor}) as a specific example, i.e.\
\begin{equation}\label{GARCHX} \zeta_t=\sigma_t \epsilon_{t+1}, \;\;\; t \in \mathbb{Z}, \end{equation} and
\begin{equation}\label{GARCHsigma} \sigma_t=\sqrt{\alpha_0+\alpha_1 \sigma_{t-1}^2 \epsilon_t^2+\beta_1 \sigma_{t-1}^2}, \;\;\; t \in \mathbb{Z},\end{equation}
for suitable constants $\alpha_0>0$ and $\alpha_1, \beta_1 \ge 0$. Here, the sequence~$(\zeta_t)_{t \in \mathbb{Z}}$ is the observable part, e.g.\ a model for financial log returns, and the series~$(\sigma_t)_{t \in \mathbb{Z}}$ describes the conditional standard deviation (volatility) of the process at time $t \in \mathbb{Z}$.
In the basic setup the innovation sequence~$(\epsilon_t)_{t \in \mathbb{Z}}$ is assumed to be i.i.d.\ standard normal. 
Note that the above GARCH(1,1) model satisfies~\eqref{Y} and~\eqref{X} for
 $$ \Phi(x,e)=\sqrt{\alpha_0+\alpha_1 x^2 e^2+\beta_1 x^2}, \;\;\; \Psi(x,e)=x e, \;\;\; s=1. $$
We remark that for $\beta_1=0$ in \eqref{GARCHsigma} the GARCH$(1,1)$ setup includes the  ARCH$(1)$ model as a special case, cf.\ \cite{engle}. For further examples, cf.\ also Remark \ref{examples}.

It is well-known~\cite{GARCH} that under quite general assumptions about the distribution of $\epsilon_t$, $ t \in \mathbb{Z}$, and about the size of the parameters $\alpha_0$, $\alpha_1$ and $\beta_1$
the stationary solutions to~\eqref{GARCHX} and~\eqref{GARCHsigma} share a common regularly varying (heavy tailed) behavior. A heavy tailed behavior of both the volatilities and the log returns is a desirable feature of financial time series as it agrees with commonly accepted stylized facts.
Accordingly, we will assume regular variation  for the stationary solutions to both $\eqref{Y}$ and $\eqref{X}$, cf.\ Condition~1 below. 
As it is not clear whether the limit in~\eqref{inquestion} exists we will discuss those questions in more detail in Sections~\ref{firstpart} and~\ref{uniqueness}. In Section~\ref{specialform} we will show that under some further assumptions the limiting distribution in~\eqref{inquestion} has a particularly simple form which can be seen as an extension to similar findings in \cite{segers}. More precisely, outside of the period $\{0,1, \dots, s\}$ our results will allow for a representation of the limit process in \eqref{inquestion} as a multiplicative random walk, cf.\ Proposition \ref{mainprop}. Heuristically, if we consider the example given by \eqref{GARCHX} and \eqref{GARCHsigma}, this is the case since a large value of $|\zeta_0|$ stems most likely from a large value of $\sigma_0$ as the tail of $\sigma_0$ is heavier than the tail of $\epsilon_1$. Now, for a large value of $\sigma_i$, $\sigma_{i+1}$ behaves asymptotically like $\sqrt{\alpha_1\epsilon_{i+1}^2+\beta_1}\sigma_i$. At the same time, for $i=1$,
 the distribution of $\epsilon_1$ is influenced by the extremal event of $|\zeta_0|=|\sigma_0 \epsilon_1|$ being large while all future $\epsilon_i, i \geq 2,$ are not influenced by this condition. In Section \ref{MRV} we will analyze connections of our results with multivariate regular variation of the time series $(X_t,Y_t)_{t \in \mathbb{Z}}$. The theoretical results are applied to the GARCH$(1,1)$ model in Section~\ref{GARCH}. They allow for a simple representation of the tail-chain in this case (cf.\ Proposition~\ref{easysimulation}) and are used for Monte Carlo evaluations of some extremal characteristics in Section~\ref{simulation}.

\section{Existence of a Limiting Distribution}\label{firstpart}
In the following, we will assume that the stationary distribution of $Y_t=\Phi(Y_{t-1},\epsilon_t)$, $t \in \mathbb{Z}$, cf.~\eqref{Y}, is regularly varying with index $\alpha>0$ and that it is tail-balanced, i.e.\ the following condition holds:
\begin{theopargself}
\begin{condition}[1.a]
\begin{equation}\label{regularvariationuniv} \lim_{x \to \infty}\frac{P(|Y_0|>ux)}{P(|Y_0|>x)}=u^{-\alpha}, \;\; \forall \, u>0, \;\;\; \lim_{x \to \infty} \frac{P(Y_0>x)}{P(|Y_0|>x)}=p \in [0,1].\end{equation} \qed
\end{condition}
\end{theopargself}
We will study the joint extremal behavior of~\eqref{Y} and~\eqref{X} under the assumption that $X_0$ shares the tail behavior of $Y_0$, i.e.\ there exists a constant $C>0$ such that
\begin{equation}\label{tailequivalence}\lim_{x \to \infty}\frac{P(|X_0|>x)}{P(|Y_0|>x)}=C.\end{equation} 

Analogous to Condition~1.a we say that \textbf{Condition~1.b} holds if the time series $(X_t)_{t \in \mathbb{Z}}$ satisfies \eqref{regularvariationuniv} with $X_0$ in place of $Y_0$ (with a possibly different value of $p$). Furthermore, if both conditions and~\eqref{tailequivalence} are satisfied we will say that \textbf{Condition~1} holds. 
\begin{proposition}\label{sequenceistight}
Let $(Y_t)_{t \in \mathbb{Z}}$ and $(X_t)_{t \in \mathbb{Z}}$ be stationary time series given by \eqref{Y} and \eqref{X} and let \eqref{regularvariationuniv} and \eqref{tailequivalence} be satisfied. Then, the family 
$$ \mathcal{L}\left(\frac{Y_m}{x}, \dots, \frac{Y_n}{x} \;\middle\vert\; |X_0|>x\right), \;\;\; x>1,$$
of conditional distributions is tight for all $-\infty<m\leq n < \infty$.
\end{proposition}
\begin{proof}
Let $u > 0$. Then
\begin{eqnarray*}
&& P \left(\bigcup_{i=m}^{n}\left\{\frac{|Y_{i}|}{x} > u\right\}\;\middle\vert\; |X_0|>x \right) \leq \sum_{i=m}^{n} P\left(\frac{|Y_{i}|}{x} > u \;\middle\vert\; |X_0|>x \right)\\
 &\leq& \sum_{i=m}^{n} \frac{P(|Y_{i}| > ux)}{P(|X_0|>x)} =\sum_{i=m}^{n} \frac{P(|Y_{0}| > ux)}{P(|Y_0|>x)}\frac{P(|Y_{0}| > x)}{P(|X_0|>x)}.
\end{eqnarray*}
By \eqref{regularvariationuniv} and \eqref{tailequivalence} the r.h.s.\ is bounded by $2(n-m+1)u^{-\alpha}C^{-1}$ for $x$ large. 
\end{proof}
Therefore, a weak accumulation point of the family of distributions exists. The following lemma shows, however, that it is not necessarily unique.
\begin{lemma}\label{nonuniquenesslemma}
There exist time series $(Y_t)_{t \in \mathbb{Z}}$ and $(X_t)_{t \in \mathbb{Z}}$ of the form \eqref{Y} and \eqref{X} such that Condition~1 is satisfied but \eqref{inquestion} has more than one weak accumulation point.  
\end{lemma}
\begin{proof}
Let $\epsilon_{t}\overset{i.i.d.}\sim \mbox{Par}(1),$ i.e.\ $P(\epsilon_t>x)=x^{-1}, x\geq 1$. With $\Phi(Y_{t-1},\epsilon_t) = \epsilon_t$ we have $Y_{t} = \epsilon_{t}, t \in 
\mathbb{Z}$, so $Y_t\overset{i.i.d.}\sim \mbox{Par}(1)$ as well. Let $s=1$ and $\Psi(Y_t,\epsilon_{t+1}) = f(\epsilon_{t+1})$ for a continuous function $f: \mathbb{R} \to \mathbb{R}$ to be described below. Thus $X_{t} = f(\epsilon_{t+1}), t \in 
\mathbb{Z}$. By independence, any weak accumulation point of $\mathcal{L}\left(\frac{Y_{0}}{x},\frac{Y_{1}}{x}\;\middle\vert\; |X_0|>x\right)$ equals $\delta_{0}\times \mu$ for some weak accumulation point $\mu$ of $\mathcal{L}\left(\frac{Y_{1}}{x}\;\middle\vert\; |X_0|>x\right)$, where $\delta_a$, $a \in \mathbb{R},$ denotes the Dirac measure in $a$. With $Y:= \epsilon_{1} \sim \mbox{Par}(1)$ we will construct $f$ such that $\mathcal{L}\left(\frac{Y}{x}\;\middle\vert\; |f(Y)|>x\right)$ has a continuum of weak accumulation points.\\
Let $f(t) = t, t \leq 1$. For the sequence $z_{i} = 5^{i}, i \in \mathbb{N}_0,$ each interval $[z_{i},5z_{i}]$ is mapped onto itself by $f$.  On  $[4z_{i},5z_{i}]$ it interpolates linearly between the values $z_{i}$ and $5z_{i}$, and on $[3z_{i},4z_{i}]$ between $3z_{i}$ and $z_{i}$. The function $f$ can be extended on each interval $[z_{i},3z_{i}]$ such that $f([z_{i},3z_{i}]) \subset [z_{i},5z_{i}] $ and $f(Y) \sim \mbox{Par}(1)$. The details of the construction of $f$ are given in the Appendix.
We first show that two different weak accumulation points exist. Since $f(Y)$ is nonnegative we drop the absolute value. Along the sequence $x_{i} =  5^{i}, i \in \mathbb{N}_0$, we have $\{f(Y) >x_{i}\} = \{Y >x_{i}\}$. Thus, for $b \geq 1$ it holds that $P\left(\frac{Y}{x_{i}}>b\;\middle\vert\; f(Y)>x_{i} \right)$ $= b^{-1}$. Hence, $\mathcal{L}\left(\frac{Y_{1}}{x_{i}}\;\middle\vert\; |X_0|>x_{i}\right) = \mbox{Par}(1)$ for all $i$.\\
Now, suppose that $x_{i} =  3\cdot5^{i}, i \in \mathbb{N}_0$. By construction $x_{i} < Y < \frac{3}{2}x_{i}$ implies $f(Y) < x_{i}$, thus $P\left(\frac{Y}{x_{i}}\in (1,\frac{3}{2})\;\middle\vert\; f(Y)>x_{i} \right) = 0$. This leads (at least along a subsequence) to a different weak limit. For $b \geq \frac{3}{2}$ one still has $P\left(\frac{Y}{x_{i}} > b \;\middle\vert\; f(Y)>x_{i} \right) = b^{-1}$, since $Y > \frac{3}{2}x_{i}$ implies $f(Y) > x_{i}$.\\
Adapting the above argument shows that each sequence $x_{i} =  c\cdot 5^{i}$,\linebreak $3 \leq c < 5$, leads to a different weak limit $\mu_{c}$  (at least along a subsequence) with $\mu_{c}((1,b_{c})) = 0$ and $\mu_{c}([b,\infty)) = b^{-1}$ for all $b \geq b_{c} = \frac{15+c}{4c}$.
\end{proof}
In order to study the properties of the limit in~\eqref{inquestion} in more detail we will make further assumptions about the functional form of $\Phi$ and $\Psi$ which relate to those given in \cite{segers}. There, the single time series $(Y_t)_{t \in \mathbb{Z}}$ is analyzed and 
both the existence and the form of the weak limit
 $$ \lim_{x \to \infty} \mathcal{L}\left(\frac{Y_{m}}{x}, \dots, \frac{Y_n}{x} \;\middle\vert\; |Y_0|>x\right) $$
for all $-\infty<m\leq n< \infty$ are discussed. 
Under Condition~1.a and under an additional assumption (cf.\ Condition~2.a below)
this so-called tail chain bears resemblance to a multiplicative random walk. The idea behind this condition and the following Proposition \ref{segers} is that for a stochastic process which behaves roughly like $Y_{t+1} \sim Y_t \cdot \phi(\epsilon_{t+1}, \sign(Y_t))$ as $|Y_t| \to \infty$ for a suitable function $\phi$, the whole process behaves like a multiplicative random walk given an extreme event at time~0.

Note that our Condition~2.a is a slightly stronger version of~\cite[Condition~2.2]{segers} that will allow to simplify some of our proofs in Section~\ref{specialform}.
\begin{theopargself}
\begin{condition}[2.a]
There exists a function $\phi: \mathbb{S}\times\{-1,1\} \to \mathbb{R}$ such that 
$$ \lim_{y \to \infty} \frac{\Phi(y w(y),v(y))}{y}=w \phi(v,\sign(w)) $$
for all $w(y)\to w \in \mathbb{R}, v(y) \to v \in \mathbb{S}$. Here, $\sign(w)=2\cdot \mathds{1}_{[0,\infty)}(w)-1,$ where $\mathds{1}_{\{\cdot\}}(\cdot)$ denotes the indicator function. \qed
\end{condition}
\end{theopargself}

This condition allows for the case $\phi(\cdot, \cdot)\equiv 0$ with trivial limit distributions of~\eqref{inquestion}. 

\begin{remark}\label{examples}
If we identify $(Y_t)_{t \in \mathbb{Z}}=(\sigma_t)_{t \in \mathbb{Z}}$ with the volatility process of a financial time series, there exist several examples which satisfy Condition~2.a:
\begin{itemize}
 \item ``standard'' GARCH(1,1) models, cf.\ \eqref{GARCHsigma}, with $\phi(\epsilon_t, \sign(\sigma_t))$\linebreak $=\sqrt{\alpha_1\epsilon_t^2+\beta_1}$, 
 \item GJR-GARCH(1,1) models (cf.\ \cite{glosten}) which reflect asymmetric behavior of the volatility process with 
 $$\sigma_t^2=\alpha_0+(\alpha_1+\delta_1 \mathds{1}_{\{\epsilon_t>0\}})\sigma_{t-1}^2\epsilon_t^2+\beta_1\sigma_{t-1}^2. $$
 Here, $\phi(\epsilon_t, \sign(\sigma_t))=\sqrt{(\alpha_1+\delta_1 \mathds{1}_{\{\epsilon_t>0\}})\epsilon_t^2+\beta_1}$,
 \item SR-SARV (stochastic volatility) models defined by $\zeta_t=\sigma_t\epsilon_{t+1}$ and volatility sequence 
  $$ \sigma_t=\alpha_0+(\gamma+\alpha_1 \sigma_{t-1})\eta_t+\beta_1\sigma_{t-1} $$
  or 
  $$ \sigma_t^2=\alpha_0+(\gamma+\alpha_1 \sigma_{t-1}^2)\eta_t+\beta_1 \sigma_{t-1}^2, $$
 where $((\eta_t,\epsilon_t))_{t \in \mathbb{Z}}$ with $\eta_t \geq 0$ is i.i.d., with a possible dependence between $\eta_t$ and $\epsilon_t$ for a fixed value of $t$ (cf.\ \cite{andersen}). (In this case the space $\mathbb{S}$ of innovations $\widetilde{\epsilon_t}=(\epsilon_t,\eta_t)$ is to be taken as $\mathbb{R}^2$.) Here, $\phi((\eta_t,\epsilon_t), \sign(\sigma_t))=\alpha_1\eta_t+\beta_1$ or $\phi((\eta_t,\epsilon_t), \sign(\sigma_t))=\sqrt{\alpha_1\eta_t+\beta_1}$, respectively.
\end{itemize}
\end{remark}
If we identify $(Y_t)_{t \in \mathbb{Z}}$ with a volatility sequence, then $Y_t \geq 0$ and the dependence of $\phi$ on $\sign(w)$ is not necessary. For general hidden Markov models, however, the extremal behavior of $Y_{t+1}$ may differ for the cases $Y_t \to \infty$ or $Y_t \to -\infty$.

The following proposition puts the aforementioned heuristic \linebreak $Y_{t+1} \sim Y_t \phi(\epsilon_t, \sign(Y_t))$ for $|Y_t| \to \infty$ on solid ground. It is taken from~\cite{segers} and will be fundamental to our subsequent analysis. Here and in the following, ``$\overset{w}{\Rightarrow}$'' denotes weak convergence of probability measures.
\begin{proposition}[cf.\ {\cite[Theorem 2.3]{segers}}]\label{segers1}
Let $(Y_t)_{t \in \mathbb{Z}}$ (not necessarily stationary) be given by~\eqref{Y} and let Conditions~1.a and~2.a hold. Then for $n \in \mathbb{N}$, as $y \to \infty$,
\begin{eqnarray*}\label{segers}
&& \mathcal{L}\left(\frac{|Y_0|}{y},\frac{Y_0}{|Y_0|},\epsilon_1,\frac{Y_1}{|Y_0|}, \dots, \epsilon_n,\frac{Y_n}{|Y_0|}\;\middle\vert\; |Y_0|>y\right)\\
&& \overset{w}{\Rightarrow} \mathcal{L}(Y, M_0, \epsilon_1^{(Y)}, M_1, \dots, \epsilon_n^{(Y)}, M_n), 
\end{eqnarray*}
with
$$ M_j=h(M_{j-1},A_j,B_j), \;\;\;\; j \in \mathbb{N},$$
where $h:\mathbb{R}^3\to \mathbb{R}, \; h(y,a,b):=y\left(a\mathds{1}_{(0, \infty)}(y)+b\mathds{1}_{(-\infty, 0)}(y)\right)$, and $Y,$ $M_0,\epsilon_1^{(Y)},$ $\epsilon_2^{(Y)},$ $\dots$ are independent with
\begin{itemize}
\item[(i)] $Y \sim \mbox{Par}(\alpha)$, i.e. $P(Y>x)=x^{-\alpha}, x \geq 1,$
\item[(ii)]  $P(M_0=1)=p=1-P(M_0=-1)$,
\item[(iii)] $\epsilon_i^{(Y)}, i \in \mathbb{N},$ are i.i.d.\ with $\mathcal{L}(\epsilon_1^{(Y)})=\mathcal{L}(\epsilon_1)$ and 
$$(A_i,B_i)=(\phi(\epsilon_i^{(Y)},1), \phi(\epsilon_i^{(Y)},-1)),\;\;\; i \in \mathbb{N}.$$
\end{itemize}
\end{proposition}
Note that by embedding the $\epsilon_i$, $i \in \mathbb{N}$, the formulation of Proposition~\ref{segers1} differs slighty from its analog in~\cite{segers}. The proof is analogous to the proof of~\cite[Theorem 2.3]{segers} and uses the continuous mapping theorem. The joint limit distribution in Proposition~\ref{segers1} will be an important building block for the derivation of~\eqref{inquestion}. But in order to derive the limit \eqref{inquestion} we also need to specify the behavior of $(Y_n)_{n \in \mathbb{Z}}$ \textit{before} the extremal event $\{|Y_0|>y\}$. Going backwards in time, things are not as simple as before. To illustrate this, think of the process $Y_t=aY_{t-1}+\epsilon_t$ with $a>0$. Now, a large value of $|Y_0|$ may either be due to a large value of $|Y_{-1}|$ or due to a large value of $|\epsilon_{0}|$. However, if we assume stationarity of $(Y_t)_{t \in \mathbb{Z}}$ in addition to the assumptions of Proposition \ref{segers1}, then note the following: For $x,y \in \mathbb{R}$ it holds that
$$ P(\min((xY_{-1})_+,(yY_0)_+)>t)=P(\min((xY_0)_+,(yY_1)_+)>t) $$
for all $t>0$, where $(x)_+:=\max(x,0)$. Now, for $|x|, |y|\leq 1$ this implies
\begin{eqnarray*}
 &&P\left(\min((xY_{-1})_+,(yY_0)_+)>t\middle\vert\; |Y_0|> t\right)\\
 &=&P\left(\min((xY_0)_+,(yY_1)_+)>t\middle\vert\; |Y_0|> t\right),
\end{eqnarray*}
where the r.h.s. converges to 
\begin{eqnarray*}
&& \lim_{t \to \infty}  P\left(\min\left(\left(x\frac{|Y_0|}{t}\frac{Y_0}{|Y_0|} \right)_+,\left(y\frac{|Y_0|}{t}\frac{Y_1}{|Y_0|}\right)_+\right)>1\middle\vert\; |Y_0|>t\right)\\
&=& P\left(\min\left(\left(xYM_0 \right)_+,\left(yYM_1\right)_+\right)>1\right)\\
&=& \int_0^1 P\left(\min\left(\left(x M_0 \right)_+^{\alpha},\left(y M_1\right)_+^{\alpha}\right)>u \right)\, du\\
&=&\int_0^\infty P\left(\min\left(\left(x M_0 \right)_+^{\alpha},\left(y M_1\right)_+^{\alpha}\right)>u \right)\, du=E\left(\min\left(\left(x M_0 \right)_+^{\alpha},\left(y M_1\right)_+^{\alpha}\right)\right),
\end{eqnarray*}
since $Y^{-\alpha} \sim \mbox{Unif}(0,1)$, where $\mbox{Unif}(0,1)$ denotes the uniform distribution on $(0,1)$ (third last equation) and $(xM_0)_+^\alpha \leq 1$ (penultimate equation). Thus, if we  assume for the moment that a limit 
$$\lim_{t \to \infty} \mathcal{L}\left(\frac{Y_{-1}}{|Y_0|}, \frac{Y_0}{|Y_0|}\middle\vert\;|Y_0|>t\right)=:\mathcal{L}(M_{-1}, M_0)$$
exists, then by a similar reasoning it holds that
\begin{equation}\label{minexpectations} E\left(\min\left(\left(x M_0 \right)_+^{\alpha},\left(y M_1\right)_+^{\alpha}\right)\right)=E\left(\min\left(\left(x M_{-1} \right)_+^{\alpha},\left(y M_0\right)_+^{\alpha}\right)\right) 
\end{equation}
for all $|x|, |y|\leq 1$. By rescaling it holds for all $x,y \in \mathbb{R}$. In fact, one can show that for all laws $\mathcal{L}(M_0,M_{1})$ which may evolve in Proposition~\ref{segers1} from a stationary Markov chain $(Y_t)_{t \in \mathbb{Z}}$ there exists a law $\mathcal{L}(M_{-1},M_0)$ satisfying \eqref{minexpectations} and that this law is sufficient to determine the whole backward limit process which is Markovian like the forward limit process. Details can be found in \cite{segers}, we state the main definitions and results below. A precise form of the limit process for the GARCH(1,1) case is given in Section \ref{GARCH}.

\begin{definition}[cf.\ \cite{segers}, Definition 4.1]\label{bftc}
A time series $(M_t)_{t \in \mathbb{Z}}$ is said to be a
back-and-forth tail chain with index $0<\alpha<\infty$ and forward transition law $\mu$, denoted by BFTC$(\alpha, \mu)$, if
\begin{itemize}
\item[(i)] $\mathcal{L}(M_0,M_1)=\mu$ with $M_0 \in \{-1,1\}, M_1 \in \mathbb{R}$,
\item[(ii)] $\mu^\ast:=\mathcal{L}(M_0,M_{-1})$ is \textit{adjoint} to $\mu$, i.e.\ 
\begin{equation}\label{adjointmeasure}E\left(\min\left((xM_0)_+^\alpha, (yM_1)_+^{\alpha}\right)\right)=E\left(\min \left(xM_{-1})_+^\alpha, (yM_0)_+^\alpha\right)\right),\;\;\; \forall\,x,y \in \mathbb{R},\end{equation} 
\item[(iii.a)] for all integer $t \geq 1$ and all real $x_{t-1}, x_{t-2},\dots$,
$$ \mathcal{L}(M_{t}|M_{t-1}=x_{t-1},M_{t-2}=x_{t-2},\dots)=\mathcal{L}(h(x_{t-1},A_1,B_1)),$$
(cf.\ Proposition \ref{segers1} for the definition of $h$) where $A_1$ and $B_1$ are independent with 
$$ \mathcal{L}(A_1)=\mathcal{L}\left(\frac{M_1}{M_0}\;\middle\vert\; M_0=1\right), \;\;\; \mathcal{L}(B_1)=\mathcal{L}\left(\frac{M_1}{M_0}\;\middle\vert\;M_0=-1\right),$$
\item[(iii.b)] for all integer $t \geq 1$ and all real $x_{-t+1}, x_{-t+2},\dots$,
$$ \mathcal{L}(M_{-t}|M_{-t+1}=x_{-t+1},M_{-t+2}=x_{-t+2},\dots)=\mathcal{L}(h(x_{-t+1},A_{-1},B_{-1})),$$
where $A_{-1}$ and $B_{-1}$ are independent with
$$ \mathcal{L}(A_{-1})=\mathcal{L}\left(\frac{M_{-1}}{M_0}\;\middle\vert\; M_0=1\right), \;\;\; \mathcal{L}(B_{-1})=\mathcal{L}\left(\frac{M_{-1}}{M_0}\;\middle\vert\;M_0=-1\right).$$
\end{itemize}
\end{definition}
\begin{proposition}[cf.\ {\cite[Theorem 5.2]{segers}}]\label{segers2}
Let $(Y_t)_{t \in \mathbb{Z}}$ be a stationary time series given by \eqref{Y} and let Conditions~1.a and~2.a hold. Then, for all $m,n \in \mathbb{N}$, as $y \to \infty$,
\begin{equation}
\mathcal{L}\left(\frac{|Y_0|}{y},\frac{Y_{-m}}{|Y_0|}, \dots, \frac{Y_n}{|Y_0|}\;\middle\vert\; |Y_0|>y\right)\overset{w}{\Rightarrow} \mathcal{L}(Y, M_{-m}, \dots, M_n), 
\end{equation}
with
\begin{itemize}
\item[(i)] $Y \sim \mbox{Par}(\alpha),$ independent of $(M_t)_{t \in \mathbb{Z}}$,
\item[(ii)] $(M_t)_{t \in \mathbb{Z}}$ is a BFTC$(\alpha, \mu)$ where $\mu=\mathcal{L}(M_0,M_1)$ with
$(M_0,M_1)$ as in Proposition \ref{segers1}.
\end{itemize}
\end{proposition}
Propositions \ref{segers1} and \ref{segers2} show that the assumption about the asymptotic behavior of $\Phi$ leads to a very simple form of the tail process for~$(Y_t)_{t \in \mathbb{Z}}$. The key ingredient is the connection between the forward and the backward limit process which is stated in \eqref{adjointmeasure}. Although this equation is sufficient for a unique determination of $\mathcal{L}(M_0,M_{-1})$ from $\mathcal{L}(M_0,M_{1})$ it appears to be cumbersome for specific applications. It is shown in \cite[Equation (3.4)]{segers} that one may derive the law of $(M_{-1}, M_0)$ from that of $(M_0,M_1)$ by noting that 
\begin{eqnarray}
\nonumber E(f(M_{-1}/M_0)|M_0=\sigma)&=& P(M_0=\sigma)^{-1}E(f(M_0/M_1)(\sigma M_1)_+^\alpha)\\
\label{easyadjointformula}&& + \, [1-P(M_0=\sigma)^{-1}E((\sigma M_1)_+^\alpha)]f(0)
\end{eqnarray}
for integrable functions $f$, and $\sigma \in \{-1,1\}$ such that $P(M_0=\sigma)>0$.

In order to discuss similar results for the above case of two connected time series we will introduce an analogous condition for~$\Psi$.
\begin{theopargself}
 \begin{condition}[2.b]
There exists a function $\psi: \mathbb{S}^{s}\times\{-1,1\} \to \mathbb{R}$ such that 
$$ \lim_{y \to +\infty} \frac{\Psi(y w(y),v(y))}{y}=w \psi(v,\sign(w))$$ 
for all $w(y)\to w \in \mathbb{R}, v(y) \to v \in \mathbb{S}^{s}$. \qed
\end{condition}
\end{theopargself} 
This condition allows for the case $\psi \equiv 0$ which is meaningless in practice. If Conditions~2.a and~2.b hold we will say that \textbf{Condition~2} is satisfied. Note that for all examples given in Remark~\ref{examples}, Condition 2.b holds due to the multiplicative form of $X_t=\Psi(\sigma_t,\epsilon_{t+1})=\sigma_t \epsilon_{t+1}$.
\section{Uniqueness of the Weak Accumulation Point}\label{uniqueness}
In the following, we will investigate the uniqueness of the accumulation point of~\eqref{inquestion} under Conditions~1 and~2.
It will turn out that the behavior of the univariate distribution $\mathcal{L}(Y_0/x\,|\, |X_0|>x)$ as $x \to \infty$ leads to a sufficient condition. 
\begin{proposition}\label{onehelpsforall} Let $(Y_t)_{t \in \mathbb{Z}}$ and $(X_t)_{t \in \mathbb{Z}}$ be stationary time series given by \eqref{Y} and \eqref{X} and let Conditions~1 and~2 hold. Equivalent are
\begin{itemize}
 \item [(i)] the weak accumulation point of \eqref{inquestion} is unique, and 
  $$\mathcal{L}(Y^{(X)}_0):=\lim_{x \to \infty}\mathcal{L}(Y_0/x \,| \, |X_0|>x)$$
  has no mass in zero,
 \item [(ii)] there exists a weak accumulation point $\mathcal{L}(\hat{Y}^{(X)}_0)$ of $\mathcal{L}(Y_0/x\,|\, |X_0|>x)$ with $\hat{Y}^{(X)}_0\neq 0$ a.s.
\end{itemize}
\end{proposition}
Consequently, the uniqueness of the limit in \eqref{inquestion} may well be derived from \textit{any} weak accumulation point. The following lemma will be used in the proof of Proposition~\ref{onehelpsforall}. In addition, it is also of interest in its own right 
as it provides a criterion for
property (ii) of Proposition~\ref{onehelpsforall}.
\begin{lemma}\label{pointmass}
Let the assumptions of Proposition \ref{onehelpsforall} hold and let $\mathcal{L}(\hat{Y}_0^{(X)})$ be any weak accumulation point of $\mathcal{L}(Y_0/x\,|\, |X_0|>x)$. Then
 $$ P(\hat{Y}^{(X)}_0=0)=1-C^{-1}E(|\chi|^\alpha), $$
with
 $$ \chi=\chi(M_0, \epsilon_1^{(Y)}, \dots, \epsilon_s^{(Y)})=M_0\cdot \psi(\epsilon_1^{(Y)}, \dots, \epsilon_s^{(Y)}, \sign(M_0)),$$
where $M_0, \epsilon_1^{(Y)}, \dots, \epsilon_s^{(Y)}$ are defined as in Proposition \ref{segers1}, and $C$ is given by \eqref{tailequivalence}. 
\end{lemma}
\begin{proof}
For $a > 0$ we have
\begin{equation*}
P\left(x^{-1}|Y_0|>a\;\middle\vert\; |X_0|>x\right)
=P\left((ax)^{-1}|X_0|>a^{-1} \;\middle\vert\; |Y_{0}| > ax\right) \cdot \frac{P(|Y_0| > ax)}{P(|X_{0}|>x)}.
\end{equation*}
With $x \to \infty$ the second term converges to $C^{-1}a^{-\alpha}$ by Condition~1. For the first term we analyze the limit of 
\begin{equation}
\mathcal{L}\left((ax)^{-1}|X_0|  \;\middle\vert\; |Y_{0}| > ax\right)=\mathcal{L}\left((a x)^{-1}  \left|\Psi\left(a x \frac{Y_0}{a x},\epsilon_1, \dots, \epsilon_{s}\right)\right|\; \vline \; |Y_{0}| > a x \right) \label{XconditionedonY}
\end{equation}
as $x \to \infty$. By an application of the continuous mapping theorem (cf.\ \cite[Theorem 4.27]{kallenberg}) in combination with Condition~2 and Proposition~\ref{segers1} this converges to $\mathcal{L}\left(|Y \cdot \chi|\right)$. Since $Y$ is Pareto distributed and independent of $\chi$ we may rule out a point mass of $|Y \cdot \chi|$ in $1/a$, and conclude that
$$ \lim_{x \to \infty} P\left((ax)^{-1}|X_0|>a^{-1} \;\middle\vert\; |Y_{0}| > ax\right)=P\left(|Y \cdot \chi|>a^{-1}\right).$$
Therefore, for a sequence $a_n \searrow 0$ which avoids possible point masses of $\hat{Y}_0^{(X)}$ it follows that
$$P(|\hat{Y}^{(X)}_0| > a_n) = P\left(|Y \cdot \chi| > a_n^{-1}\right)\cdot \frac{a_n^{-\alpha}}{C} = \frac{P\left(Y\cdot |\chi| > a_n^{-1}\right)}{P\left(Y>a_n^{-1}\right)} \cdot \frac{1}{C}.$$ 
The result now follows with $a_n \searrow 0$ if we show that
\begin{equation}\label{kindofbreiman}\lim_{x \to \infty} \frac{P(Y\cdot |\chi|>x)}{P(Y>x)}=E(|\chi|^\alpha),\end{equation} 
which can be seen as a kind of extension of Breiman's Theorem (cf.\ \cite{breiman}) for the special case of $Y \sim \mbox{Par}(\alpha)$. It follows because $P(Y\cdot |\chi| > x)= P(Y \cdot|\chi| > x,$ $|\chi|\leq x) + P(Y\cdot |\chi| > x,|\chi| > x)$, where the first term equals $\int_0^x P(Y>\frac{x}{z})dP^{|\chi|}(z)=\int_0^x\left(\frac{z}{x}\right)^\alpha dP^{|\chi|}(z).$ The second term equals $P(|\chi|>x)$, since $Y\geq 1$ a.s. Thus, $P(Y \cdot |\chi|>x)/P(Y>x)$ $=x^\alpha P(Y \cdot |\chi|>x)$ $=\int_0^\infty z^\alpha \mathds{1}_{[0,x]}(z)+x^\alpha \mathds{1}_{(x, \infty)}(z)dP^{|\chi|}(z)$ and \eqref{kindofbreiman} follows from monotone convergence. This gives the result.
\end{proof}
\begin{proof}[Proof of Proposition \ref{onehelpsforall}]
We show that (ii) implies (i). Let $\nu_{1}$ and $\nu_{2}$ denote two weak accumulation points. In the following, let $a_0 \geq 0, a_{1},\ldots,a_s \in \mathbb{R}$, and $A = A(a_{1},\ldots,a_s) := (a_{1},\infty)\times \ldots \times (a_s,\infty)$. We will show that $\nu_{k}((a_0,\infty) \times A)$ and $\nu_{k}([-a_0,\infty) \times A)$ do not depend on $k \in \{1,2\}$. Here, we shall use that  (ii) implies $C = E(|\chi|^{\alpha})$ by Lemma \ref{pointmass}, which in turn implies that $\nu(\{0\} \times \mathbb{S}^{s}) =0$ for any weak accumulation point $\nu$. Since the above sets form a generating $\pi$-system, any two weak accumulation points coincide. By tightness (cf.\ Proposition \ref{sequenceistight}) this implies weak convergence.\\
Consider first $a_0>0$ and $a_{1},\ldots,a_s\neq 0$ that avoid the at most countably many point masses of the coordinate projections of $\nu_{1}$ and $\nu_{2}$. Then, $\nu_{k}((a_0,\infty) \times A)$ is the limit of 
\begin{equation}\label{noabsolutvalue}P\left(\frac{Y_0}{x}>a_0,\frac{Y_1}{x}>a_{1},\ldots,\frac{Y_s}{x}>a_{s}\;\middle\vert\; |X_0|>x\right)
\end{equation}
along a subsequence depending on $k$.  For general $x$, insert $\frac{|Y_0|}{x}>a_0$ in \eqref{noabsolutvalue}. This probability equals 
\begin{eqnarray*}
&& P\left(\frac{Y_{0}}{a_0x}>1,\frac{Y_{1}}{a_0x}>\frac{a_{1}}{a_0},\ldots,\frac{Y_s}{a_0x}>\frac{a_s}{a_0}, \frac{|X_{0}|}{a_0x}>\frac{1}{a_0} \;\middle\vert\; |Y_{0}| > a_0x\right)\\
&& \hspace{1cm}\cdot \frac{P(|Y_{0}| > a_0x)}{P(|X_{0}|>x)}\,.
\end{eqnarray*}
By Conditions~1 and~2, and since the variables have point masses at most at zero, this converges to 
$$P\left(Y  M_0 > 1,Y M_1 > \frac{a_{1}}{a_0},\ldots, Y  M_{s}>\frac{a_s}{a_0}, Y|\chi| > \frac{1}{a_0} \right)\cdot \frac{1}{C}a_0^{-\alpha},$$
cf.\ the proof of Lemma \ref{pointmass}.

We have shown that  $\nu_{k}((a_0,\infty) \times A)$ does not depend on $k$. Approximation from inside extends this to all $a_0 \geq 0$ and $a_{1},\ldots,a_{s} \in \mathbb{R}$.
Replacing $\frac{Y_{0}}{x}>a_0$ by $\frac{Y_{0}}{x}<-a_0$ the same computation followed by an approximation argument shows the same for $\nu_{k}((-\infty,-a_0)\times A)$. Combining these two results for $a_0=0$ with $\nu_{k}(\{0\} \times \mathbb{R}^{s}) =0$ shows that $\nu_{k}(\mathbb{R} \times A)$ does not depend on $k \in \{1,2\}$. Thus, the same holds for the sets $[-a_0,\infty) \times A = (\mathbb{R} \times A) \setminus ((-\infty,-a_0) \times A)$.
%
\end{proof}
\begin{remark} Lemma \ref{pointmass} shows that Prop.~\ref{onehelpsforall} (ii), and thus (i), holds if and only if $C = E(|\chi|^{\alpha})$. We give some examples. Suppose that $(X_{t})_{t \in \mathbb{Z}}$ and $(Y_{t})_{t \in \mathbb{Z}}$ are nonnegative time series and $X_{t} = Y_{t} \cdot \psi(\epsilon_{t+1})$, thus  $\chi = \psi$. Then, \linebreak $E(\psi(\epsilon_{0})^{\alpha+\delta})<\infty$ for some $\delta >0$ implies $C = E(|\chi|^\alpha) < \infty$ by Breiman's Theorem (cf.\ \cite{breiman}) and Condition~2 holds if  $\psi$ is continuous. If $P(Y_{0} >x) \sim c\cdot x^{-\alpha}$ for some $c > 0$, then $E(\psi(\epsilon_{0})^{\alpha}) < \infty$ suffices to derive the same result (cf.\ e.g.\ \cite[Lemma 2.1]{gomes}). For the special case $Y_{0} \sim Par(\alpha)$ cf.\ the end of the proof of Lemma~3.2. For further generalizations of Breiman's Theorem see \cite{denisov}.
\end{remark}
\begin{remark}
By similar computations it can be shown that under the assumptions of Proposition \ref{onehelpsforall} uniqueness of the weak limit in \eqref{inquestion} is also ensured by $P(M_{-1}=0)=0$, with $M_{-1}$ as in Proposition \ref{segers2}. A key step in the argument shows that this condition implies weak convergence of 
$$\mathcal{L}\left(\frac{Y_{-m}}{y},\epsilon_{-m+1}, \dots , \frac{Y_{-1}}{y},\epsilon_{0},\frac{Y_{0}}{y},\epsilon_{1},\frac{Y_{1}}{y},\ldots,\epsilon_{n},\frac{Y_{n}}{y}\;\middle\vert\; |Y_0|>y\right)$$
as $y \to \infty$ for all $m \geq 1$ and $n \geq 0$. We give an example with $P(M_{-1}=0) = 0$ but $P(Y^{(X)}_{0}=0)> 0$, i.e.\ $P(M_{-1}=0) = 0$ may ensure uniqueness even if property (ii) in Proposition \ref{onehelpsforall} fails. To this end, let $Y_{0}$ and $\epsilon_{t},  t \in \mathbb{Z},$ be nonnegative i.i.d.\ random variables with $P(Y_{0} > x) = x^{-1}\ln(x)^{-2}$ for  $x \geq c$, where  $c \approx 2.02$ solves $x \cdot \ln(x)^2 = 1$. With $\Phi(y,v) = y$, let $Y_{t} = Y_{0}$ for all $t \in  \mathbb{Z}$. Then, $Y_{-1}=Y_{0}$ implies $Y\cdot M_{-1} \sim Y \sim\, $Par$(1)$, thus $P(M_{-1} = 0) = 0$. For $s_{-}=-1,s_{+}=1$ let $X_{t}=Y_{t}\cdot \epsilon_{t+1},  t \in \mathbb{Z}$. Careful calculations show that $C = \lim_{x\to \infty}\frac{P(X_{0} > x)}{P(Y_{0}> x)} = 2(c+\sqrt{c})$ (cf.\ \cite{denisov} for similar arguments). But with $\alpha =1$ it holds that $E(|\chi|^{\alpha}) = E(\epsilon_{1}) = \int_{0}^{\infty}P(\epsilon_{1} > x)dx = c +\frac{1}{\ln(c)} = c + \sqrt{c}$, thus $P(Y^{X}_0=0)=1/2$ 
by Lemma \ref{pointmass}.
\end{remark}
\section{Structure of the Limit Process}\label{specialform}
While the existence of a limit in \eqref{inquestion} has been analyzed in the preceding section we will now deal with the particular form of the limit.
For easy reference we shall introduce the following condition.\\
\begin{theopargself}
 \begin{condition}[3]
There exists a random vector $(Y^{(X)}_0, \dots, Y^{(X)}_s)$ such that
\begin{equation*}
\lim_{x \to \infty} \mathcal{L}\left(\frac{Y_0}{x}, \dots, \frac{Y_{s}}{x}\;\middle\vert\; |X_0|>x\right)=\mathcal{L}(Y^{(X)}_0, \dots, Y^{(X)}_{s}).
\end{equation*} \qed
\end{condition}
\end{theopargself}

We assume that the limit distribution in Condition 3 is unique in order to simplify the statement of the proposition below. Note, however, Remark \ref{hasnottobeunique} at the end of this section for a generalization to the case of non-uniqueness.
We will use Conditions~1 to~3 to derive a result for the form of the limit in \eqref{inquestion} which is similar to Proposition \ref{segers2}. 

While Conditions~1 and~2 bear a natural resemblance to the assumptions made in \cite{segers}, Condition~3 is necessary to ensure that a ``starting point'' for a tail chain exists that covers the time span from $0$ to $s$ where the $\epsilon_1, \dots, \epsilon_{s}$ and therefore $Y_{0}, \dots, Y_{s}$ are directly influenced by the event $\{|X_0|>x\}$. We will see that outside of this range the behavior of the process $(Y_t)_{t \in \mathbb{Z}}$ corresponds to Proposition \ref{segers2}.
\begin{proposition}\label{mainprop} Let $(Y_t)_{t \in \mathbb{Z}}$ and $(X_t)_{t \in \mathbb{Z}}$ be stationary time series given by \eqref{Y} and \eqref{X} and let Conditions~1, 2 and~3 hold.
Then, for all integers $m \geq 0$ and $n \geq 0$ we have 
\begin{eqnarray}\label{E:mainprop}\lim_{x \to \infty}\mathcal{L}\left(\frac{Y_{-m}}{x},\dots,\frac{Y_{s+n}}{x} \;\middle\vert\; |X_0|>x\right)
&=&\mathcal{L}(Y^{(X)}_{-m},\dots,Y^{(X)}_{s+n})\end{eqnarray}
with $(Y^{(X)}_0,\dots, Y^{(X)}_{s})$ as in Condition~3, and
\begin{align*} 
	Y_t^{(X)} &= h(Y_{t-1}^{(X)},A_t,B_t), &t&>s, \\
	Y_{-t}^{(X)} &= h(Y_{-t+1}^{(X)},A_{-t}, B_{-t}), &t&> 0,
\end{align*}
cf.\ Proposition \ref{segers1} for the definition of $h$. Here, $(A_t,B_t), t \in \mathbb{Z},$ are independent, and independent of $(Y_0^{(X)}, \dots, Y_{s}^{(X)})$ with
$$ \mathcal{L}(A_t,B_t)=\mathcal{L}(A_1,B_1), \; t \geq 1, \;\;\, \mathcal{L}(A_t,B_t)=\mathcal{L}(A_{-1},B_{-1}), \, t \leq -1.$$
Further, $\mathcal{L}(A_1,B_1)$ and $\mathcal{L}(A_{-1},B_{-1})$ are as in Definition \ref{bftc}.
\end{proposition}
The proof is predecessed by a lemma and a corollary where we only assume that Conditions~1 and~2 hold.
\begin{lemma}\label{etadeltalemma}
Let $m\geq 0$. For any $\eta>0$ there is $\delta_0(\eta)>0$ such that for $x$ large enough
 $$ P\left(\frac{|Y_{-m}|}{x}>\eta, \frac{|Y_{0}|}{x} \leq \delta\;\middle\vert\; |X_0|>x\right)<\eta $$ 
for all $\delta<\delta_0(\eta)$. 
\end{lemma}
\begin{proof} For $m=0$ the statement follows with $\delta_0(\eta)=\eta$. So assume that $m>0$. 
The l.h.s.\ equals 
$$P\left(\frac{|Y_{0}|}{x} \leq \delta,  |X_{0}| > x  \; \vline \; \frac{|Y_{-m}|}{x} > \eta \right)\cdot \frac{P(|Y_{-m}|>\eta x)}{P(|X_{0}| > x)}.$$
The second factor converges to $C^{-1}\eta^{-\alpha}$ by Condition~1. It suffices to show that the first factor becomes small for $\delta \to 0$.
To this end, note that by stationarity the first factor equals
$$ P\left(\frac{|Y_{m}|}{\eta x} \leq \frac{\delta}{\eta},  \frac{|X_{m}|}{\eta x} > \frac{1}{\eta}  \; \vline \; |Y_{0}| > \eta x \right) $$
which by definition of $X_{m}$ equals 
$$ P\left(\frac{|Y_{m}|}{\eta x} \leq \frac{\delta}{\eta},  \frac{\left|\Psi\left(\eta x \frac{Y_{m}}{\eta x},\epsilon_{m+1}, \dots, \epsilon_{m+s}\right)\right|}{\eta x} > \frac{1}{\eta}  \; \vline \; |Y_{0}| > \eta x \right).$$
We proceed as in the proof of Lemma \ref{pointmass}. By an application of the continuous mapping theorem with Condition~2 and Proposition \ref{segers1} this converges to \linebreak $P\left(Y|M_{m}| \leq \delta / \eta, Y |\chi_{m}| >  1 /\eta\right)$ with 
$$  \chi_{m}:= M_{m} \psi(\epsilon_{m+1}^{(Y)}, \dots, \epsilon_{m+s}^{(Y)},\sign(M_{m})).$$
Again, we use that the two limit random variables include $Y \sim \mbox{Par}(\alpha)$ as an independent factor
which excludes point masses on the positive axis. Now, the set $\left\{Y|M_{m}| \leq \delta / \eta, Y |\chi_{m}| > 1 /\eta\right\}$ is contained in
$$
 \left\{|\chi_{m}|/|M_{m}| > 1 /\delta\right\}=\left\{|\psi(\epsilon_{m+1}^{(Y)}, \dots, \epsilon_{m+s}^{(Y)},\sign(M_{m}))| > 1 /\delta\right\}.
$$
For $0<\delta<\delta_0(\eta)$ the probability of this event gets arbitrarily small for $\delta_0(\eta)$ small enough.
\end{proof}
\begin{corollary}\label{makesmainproofeasier}
Let $m, n\geq 0$ and $f$ be a bounded uniformly continuous function on $\mathbb{R}^{n+1}$ with $f(0,\ldots) = 0$. For any $\epsilon > 0$ there is $\delta_0(\epsilon) > 0$ such that for $x$ large enough 
 \begin{equation}\label{Cor4.3}
E\left(f\left(\frac{Y_{-m}}{x},\ldots,\frac{Y_{-m+n}}{x}\right)\cdot \mathds{1}_{\{|Y_{0}| \leq \delta x\}} \; \vline \; |X_{0}| > x\right) < \epsilon
\end{equation}
for all $0<\delta<\delta_0(\epsilon)$.
\end{corollary}
\begin{proof}
Since $f$ is bounded and uniformly continuous with $f(0,\ldots) = 0$, there is some $\eta >0$ such that
 \begin{equation*}
||f||_{\infty}\cdot \eta + \sup\{|f(y_{-m},\ldots,y_{-m+n})| \mbox{ with }|y_{-m}| \leq \eta \} < \epsilon.
\end{equation*}
Choose $\delta_0$ as in Lemma \ref{etadeltalemma}. For $\delta<\delta_0$ split the expected value in \eqref{Cor4.3} into two by splitting $\mathds{1}_{\{|Y_{0}| \leq \delta x\}}$ into 
$$\mathds{1}_{\{|Y_{-m}|> \eta x, |Y_{0}| \leq \delta x\}} + \mathds{1}_{\{|Y_{-m}|\leq \eta x,|Y_{0}| \leq \delta x\}}.$$ 
The first expected value is bounded by $||f||_{\infty}\cdot \eta$ by Lemma~\ref{etadeltalemma}, and the second by  
$$\sup\{|f(y_{-m},\ldots,y_{-m+n})| \mbox{ with } |y_{-m}| \leq \eta \}.$$
\end{proof}
\begin{proof}[Proof of Proposition \ref{mainprop}]
Note that the case $m=0$ and $n \geq 0$ is analogous to the proof of Proposition \ref{segers1} (cf.\ \cite[Theorem 2.3]{segers}). Since $(\epsilon_{s+1},\epsilon_{s+2}, \dots)$ is independent of $(X_0, Y_{0}, \dots,Y_{s})$ the continuous mapping theorem can be applied to derive \eqref{E:mainprop} and leads to the multiplicative structure with independent increments. 

Let now $m\geq 1$ and $n\geq 0$, and let us assume that Proposition \ref{mainprop} holds for $(Y^{(X)}_{-m+1}, \dots, Y^{(X)}_{s+n})$. Let $f:\mathbb{R}^{s+m+n+1} \to \mathbb{R}$ be bounded and uniformly continuous. We will show that
\begin{equation}\label{whattoshow}
\lim_{x \to \infty} E\left(f\left(\frac{Y_{-m}}{x}, \dots, \frac{Y_{s+n}}{x}\right)\;\middle\vert\; |X_0|>x\right)
 =E\left(f(Y_{-m}^{(X)}, \dots, Y_{s+n}^{(X)})\right) 
\end{equation}
with $(Y_{-m}^{(X)}, \dots, Y_{s+n}^{(X)})$ as defined in the statement of the proposition. Let us further assume that $f(x, \ldots)=0$ as soon as $x=0$. Note that an arbitrary function $f:\mathbb{R}^{s+m+n+1} \to \mathbb{R}$ can be split up additively into two functions $f_1$ and $f_2$ with
\begin{eqnarray*}
f_1(x_{-m}, \dots, x_{s+n})&=&f(0,x_{-m+1}, \dots, x_{s+n}),\\
f_2(x_{-m}, \dots, x_{s+n})&=&f(x_{-m}, \dots, x_{s+n})-f_1(x_{-m}, \dots, x_{s+n}),
\end{eqnarray*}
such that the second function satisfies the aforementioned assumption and the first function depends merely on $(x_{-m+1}, \dots,$ $x_{s+n})$. Since the induction hypothesis implies that \eqref{whattoshow} is satisfied by a function of $(x_{-m+1}, \dots, x_{s+n})$ the assumption about the structure of $f$ is no loss of generality.

The idea of the proof is to substitute the condition $\{|X_0|>x\}$ by a corresponding event in~$(Y_t)_{t \in \mathbb{Z}}$.
Let $\epsilon>0$. Then, for $x$ large enough
\begin{eqnarray*}
&&  \left|E\left(f\left(\frac{Y_{-m}}{x}, \dots, \frac{Y_{s+n}}{x}\right)\;\middle\vert\; |X_0|>x\right)\right. \\ 
&& \left. -E\left(f\left(\frac{Y_{-m}}{x}, \dots, \frac{Y_{s+n}}{x}\right)\mathds{1}_{\{|Y_0|>\delta x\}}\;\middle\vert\; |X_0|>x \right)\right|<\epsilon,
\end{eqnarray*}
for all $0<\delta<\delta_0(\epsilon)$, where $\delta_0(\epsilon)$ is chosen according to Corollary \ref{makesmainproofeasier}. We have
\begin{eqnarray*}
&& E\left(f\left(\frac{Y_{-m}}{x}, \dots, \frac{Y_{s+n}}{x}\right)\mathds{1}_{\{|Y_{0}|>\delta x\}}\;\middle\vert\; |X_0|>x\right)\\
&=& \frac{P(|Y_{0}|>\delta x)}{P(|X_0|>x)}E\left(f\left(\frac{Y_{-m}}{x}, \dots, \frac{Y_{s+n}}{x}\right) \mathds{1}_{\{|X_0|>x\}}\;\bigg\vert\; |Y_{0}|>\delta x \right)\\
&=& \frac{P(|Y_{0}|>y)}{P(|X_0|>y/\delta)}E\left(f\left(\delta \frac{Y_{-m}}{y}, \dots, \delta \frac{Y_{s+n}}{y}\right) \mathds{1}_{\{\delta \cdot |\Psi(Y_{0}, \epsilon_1, \dots, \epsilon_{s})|>y\}}
\;\bigg\vert\; |Y_0|>y\right)
\end{eqnarray*}
with the substitution $y=\delta x$. Here, the first factor converges by Condition~1. Furthermore, an application of the continuous mapping theorem in connection with Propositions~\ref{segers1} and~\ref{segers2} yields that the whole expression converges to
\begin{equation*}
\frac{\delta^{-\alpha}}{C}E\left(f\left(\delta Y M_{-m}, \dots, \delta Y M_{s+n}\right) \mathds{1}_{\{\delta \cdot |Y M_{0} \psi(\epsilon_1^{(Y)}, \dots, \epsilon_{s}^{(Y)})|>1\}}
\right)\end{equation*}
with $Y,$ $\epsilon_i^{(Y)}, i \in \mathbb{N},$ and $M_n, n \in \mathbb{Z},$ as in Propositions~\ref{segers1} and~\ref{segers2}. Defining new variables $(\widetilde{A}_{-m}, \widetilde{B}_{-m})$ with the same distribution as $(A_{-m}, B_{-m})$ in the statement of the proposition and independent of $Y, \epsilon_1^{(Y)}, \dots, \epsilon_{s}^{(Y)},$ $M_{-m+1}, \dots,$ $M_{s+n}$, $(Y_t)_{t \in \mathbb{Z}}, (X_t)_{t \in \mathbb{Z}}$ the above expression equals
$$
\frac{\delta^{-\alpha}}{C}E\biggl(f\left(h(\delta Y M_{-m+1}, \widetilde{A}_{-m}, \widetilde{B}_{-m}),\dots,\delta Y M_{s+n}\right)\mathds{1}_{\{\delta \cdot |Y M_{0} \psi(\epsilon_1^{(Y)}, \dots, \epsilon_{s}^{(Y)})|>1\}}
\biggr)$$
by the definition of $M_{-m}$. Next, note that by the continuous mapping theorem this equals
$$
\lim_{y \to \infty}\frac{\delta^{-\alpha}}{C}E\biggl(f\left(h\left(\delta \frac{Y_{-m+1}}{y}, \widetilde{A}_{-m}, \widetilde{B}_{-m}\right),\dots, \delta \frac{Y_{s+n}}{y} \right)\mathds{1}_{\{\delta \cdot \left| X_{0}\right|>y\}}
\;\Biggr\vert\; |Y_0|>y \Biggr).$$
Replacing $y$ by $\delta x$ and again using Condition~1 this becomes
$$
\lim_{x \to \infty}E\Biggl(f\left(h\left(\frac{Y_{-m+1}}{x}, \widetilde{A}_{-m}, \widetilde{B}_{-m}\right), \dots, \frac{Y_{s+n}}{x}\right) \mathds{1}_{\{|Y_{0}|> \delta x\}}
\;\Biggr\vert\;|X_0|>x  \Biggr). $$
Since both $h$ and $f$ are uniformly continuous with $f(0, \dots)=0$ and $h(0, \dots)=0$ this gives
\begin{eqnarray*}&&\lim_{x \to \infty}E\Biggl(f\left(h\left(\frac{Y_{-m+1}}{x}, \widetilde{A}_{-m},\widetilde{B}_{-m}\right),\dots, \frac{Y_{s+n}}{x} \right)\mathds{1}_{\{|Y_{0}|\leq \delta x\}}
\;\Biggr\vert\;|X_0|>x  \Biggr) \\
&=& \lim_{x \to \infty}E\Biggl(g\left(\frac{Y_{-m+1}}{x},\dots, \frac{Y_{s+n}}{x} \right)\mathds{1}_{\{|Y_{0}|\leq \delta x\}}
\;\Biggr\vert\;|X_0|>x  \Biggr)
\end{eqnarray*}
for the complementary expression, where
$$ g(y_{-m+1}, \ldots, y_{s+n}):=E(f(h(y_{-m+1},\widetilde{A}_{-m},\widetilde{B}_{-m}),\ldots, y_{s+n})) $$
with $g(0, \ldots)=0$. We may thus conclude from Corollary \ref{makesmainproofeasier} that
$$ \lim_{x \to \infty}E\Biggl(f\left(h\left(\frac{Y_{-m+1}}{x}, \widetilde{A}_{-m},\widetilde{B}_{-m}\right),\dots, \frac{Y_{s+n}}{x} \right)\mathds{1}_{\{|Y_{0}|\leq \delta x\}}
\;\Biggr\vert\;|X_0|>x  \Biggr) $$
tends to 0 as $\delta \to 0$. Thus, 
\begin{eqnarray*}
 && \lim_{x \to \infty} E\left(f\left(\frac{Y_{-m}}{x}, \dots, \frac{Y_{s+n}}{x}\right)\;\middle\vert\; |X_0|>x\right) \\
 &=& \lim_{x \to \infty}E\Biggl(f\left(h\left(\frac{Y_{-m+1}}{x}, \widetilde{A}_{-m}, \widetilde{B}_{-m}\right), \dots,\frac{Y_{s+n}}{x}\right)  \;\Biggr\vert\;|X_0|>x  \Biggr).
\end{eqnarray*}
An application of the continuous mapping theorem in connection with the induction hypothesis yields that the latter expression equals
$$ E\Biggl(f\left(h\left(Y_{-m+1}^{(X)}, A_{-m}, B_{-m}\right), \dots, Y_{s+n}^{(X)}\right)\Biggr), $$
with $(A_{-m}, B_{-m})$ as in the statement of the proposition. Since $Y_{-m}^{(X)}$\linebreak  $=h\left(Y_{-m+1}^{(X)}, A_{-m}, B_{-m}\right)$ this finishes the proof.
\end{proof}
\begin{remark}\label{hasnottobeunique} If $(\hat{Y}^{(X)}_0, \dots, \hat{Y}^{(X)}_{s})$ is a random vector such that for a sequence $(x_n)_{n \in \mathbb{N}}$ with $x_n\to \infty$ the relation 
$$ \lim_{n \to \infty} \mathcal{L}\left(\frac{Y_0}{x_n}, \dots, \frac{Y_{s}}{x_n}\;\middle\vert\; |X_0|>x_n\right)=\mathcal{L}(\hat{Y}^{(X)}_0, \dots, \hat{Y}^{(X)}_{s})$$
holds instead of Condition~3 then a statement
analogous to Proposition \ref{mainprop} holds true along the sequence $(x_n)_{n \in \mathbb{N}}$. The existence of such sequences is guaranteed by Condition~1, cf.\ Proposition \ref{sequenceistight}.
\end{remark}
\begin{remark}
In order to simplify notation (using only $s$ instead of $s_-$ and $s_+$), we have assumed that \eqref{X} holds instead of \eqref{Xlonger}. However, under assumption \eqref{Xlonger} the statement of Proposition \ref{mainprop} looks very similar, cf.\ \cite{janssen}, Theorem 3.5.2, for details.
\end{remark}

\section{Multivariate Regular Variation}\label{MRV}
In this chapter we will show that Condition~3 is closely related to the theory of multivariate regular variation. In a time series context this property is well explored in the case of GARCH$(p,q)$ processes, cf.\ \cite{GARCH}. 

From the equivalent definitions of multivariate regular variation given in the literature we shall refer to the one used in \cite{segers}. Recall that a measurable function $U: \mathbb{R}_+ \to \mathbb{R}_+$ is said to be univariate regularly varying with index $\alpha \in \mathbb{R}$ if $\lim_{x \to \infty}U(\tau x)/U(x)=\tau^\alpha$ for all $\tau>0$. We call a random vector $\mathbf{Z} \in \mathbb{R}^d$ multivariate regularly varying if there exists a univariate regularly varying function $U:\mathbb{R}_+\to\mathbb{R}_+$ with index $-\alpha$ and a non-degenerate, non-zero Radon measure $\nu$ on $\mathbb{E}=[-\infty,\infty]^d\setminus\{\mathbf{0}\}$ such that
\begin{equation}\label{regularvariation1}   P\left(\mathbf{Z}\in x \cdot\right)/U(x)\overset{v}{\Rightarrow}\nu(\cdot), \qquad x\to \infty, \end{equation}
where ``$\overset{v}{\Rightarrow}$'' stands for vague convergence (cf.\ \cite{phenomena}) in $M_+(\mathbb{E})$, the space of all nonnegative Radon measures on $\mathbb{E}$. One can show that the limit measure~$\nu$ is necessarily homogeneous, i.e.\ that $\nu (x A) = x^{-\alpha} \nu (A)$ holds for all $x>0$ and for all Borel sets $A \subset \mathbb{E}$ (cf.\ \cite{phenomena}).  The measure $\nu$ and, consequently, the extremal behavior of~$\mathbf{Z}$ are thus completely described by the index $\alpha$ of regular variation, a constant $c>0$ and a probability measure $S$ on $\mathbb{S}^{d-1}:=\{x \in \mathbb{R}^d| \| x \|=1\}$. The latter is the so-called spectral measure. Altogether, we have that
 $$\nu\left( \left\{x \in \mathbb{E}: \|x\|>a, \frac{x}{\|x\|} \in \cdot \right\}\right)=c \cdot a^{-\alpha}\cdot S(\cdot) $$
holds for all $a>0$ (cf.\ \cite{phenomena}).  


It has been shown by \cite{GARCH} and \cite{boman} (cf.\ also \cite{handbook}) that under mild assumptions about the distribution of $\epsilon_t, t \in \mathbb{Z},$ a stationary GARCH$(p,q)$ process is multivariate regularly varying, i.e.\ for $m,n \geq 0$ the vector $\mathbf{Z}=(\sigma_{-m}^2,\zeta_{-m}^2, \dots, \sigma_n^2, \zeta_n^2)$ with $\sigma_t$ and $\zeta_t$ as defined in \eqref{GARCHX} and \eqref{GARCHsigma} satisfies~\eqref{regularvariation1}. Furthermore, one can easily show that the same holds for the vector~$(\sigma_{-m},|\zeta_{-m}|, \dots, \sigma_n, |\zeta_n|)$. Now, the fact that a certain vector derived from the processes $(Y_t)_{t \in \mathbb{Z}}$ and $(X_t)_{t \in \mathbb{Z}}$ is multivariate regularly varying will be useful in the verification of Condition~3 as is shown in the following. 

Let us again assume that $(Y_t, X_t)_{t \in \mathbb{Z}}$ is stationary and given by \eqref{Y} and \eqref{X}. Note that Condition~3 is equivalent to
\begin{eqnarray}\nonumber && \lim_{x \to \infty} P\left(\left(\frac{Y_0}{x}, \dots, \frac{Y_s}{x}\right) \in A \;\middle\vert\; |X_0|>x\right)\\
&=&\label{lookslikeRV} \lim_{x \to \infty} \frac{P\left(\left(\frac{Y_0}{x}, \dots, \frac{Y_{s}}{x}\right) \in A, |X_0|>x\right)}{P(|X_0|>x)}\\
\nonumber &=& P\left( \left( Y_{0}^{(X)}, \dots, Y_{s}^{(X)}\right)\in A\right) 
\end{eqnarray}
for a random vector $(Y_{0}^{(X)}, \dots, Y_{s}^{(X)})$ and for all $A \in {\mathbb{B}}^{s+1}$ such that \linebreak $P((Y_{0}^{(X)}, \dots, Y_{s}^{(X)}) \in \partial A)=0$. 

In the following
we will assume multivariate regular variation of \linebreak $(|X_0|, Y_{0}, \dots, Y_{s})$ on $\mathbb{C}=\left(\bar{\mathbb{R}}_{+,0}\times\bar{\mathbb{R}}^{s+1}\right)\setminus\{\mathbf{0}\}$,
and show how this concept
relates strongly to Condition~3. 
By continuity from below it suffices to look at such~$A$ which are bounded away from $\mathbf{0}$ in order to derive Condition~3 from \eqref{lookslikeRV}. The assumption of multivariate regular variation of $(|X_0|, Y_{0}, \dots, Y_{s})$ guarantees the existence of a function $U:\mathbb{R}_+\to\mathbb{R}_+$ such that
\begin{eqnarray}
&& \nonumber \lim_{x \to \infty} \frac{P\left(|X_0|>x,\left(\frac{Y_{0}}{x}, \dots, \frac{Y_{s}}{x}\right) \in A\right)}{P(|X_0|>x)}\\
&=& \nonumber \lim_{x \to \infty} \frac{P\left(\left(\frac{Y_0}{x}, \dots, \frac{Y_{s}}{x}\right) \in A, |X_0|>x\right)}{U(x)}\frac{U(x)}{P(|X_0|>x)}\\
&=& \label{3.3followsfromRValmost} \frac{\nu((1, \infty) \times A)}{\nu((1, \infty) \times \mathbb{R}^{s+1})} 
\end{eqnarray}
if the denominator is positive (it is necessarily finite since $(1, \infty) \times \mathbb{R}^{s+1}$ is bounded away from the origin). One easily checks that \eqref{3.3followsfromRValmost} defines a probability measure for $A \in \mathbb{B}^{s+1}$ and may be set as the law of the random vector $(Y_{0}^{(X)}, \dots, Y_{s}^{(X)})$ if $\nu((1,\infty)\times \mathbb{R}^{s+1})>0$. Because of the aforementioned homogeneity of $\nu$ we note the equivalence
\begin{equation}\label{hyperplane} \nu((1,\infty)\times \mathbb{R}^{s+1})=0 \Leftrightarrow \nu((\delta,\infty)\times \mathbb{R}^{s+1})=0 \,\,\, \forall \, \delta>0. \end{equation}
Thus, $\nu((1,\infty)\times \mathbb{R}^{s+1})=0$ implies that the mass of $\nu$ is concentrated on the hyperplane $\{0\} \times \mathbb{R}^{s+1}$. Note that this is not excluded by our definition of regular variation. Nevertheless, since $\nu$ is non-degenerate and since the process $(Y_t)_{t \in \mathbb{Z}}$ is stationary we find that 
$$ \nu((1,\infty)\times \mathbb{R}^{s+1})=0 \Rightarrow \lim_{x \to \infty} \frac{P(|Y_0|>x)}{U(x)}> 0.$$
On the other hand, $\nu((1,\infty)\times \mathbb{R}^{s+1})=0$ implies that
$$  \lim_{x \to \infty} \frac{P(|X_0|>x)}{U(x)}=0. $$
Hence, $\nu((1,\infty)\times \mathbb{R}^{s+1})=0$ entails that $|X_0|$ and $|Y_0|$ are not tail equivalent. 
This contradicts Condition~1 and leads to the following proposition.
\begin{proposition}\label{mrvandcond3} Let $(|X_0|, Y_{0}, \dots, Y_{s}) \in \mathbb{R}_{+,0} \times \mathbb{R}^{s+1}$ be a multivariate regularly varying vector with index $\alpha$ and let Condition~1 hold. Then Condition~3 is satisfied. 
\end{proposition}

\section{Application to GARCH$(1,1)$ Processes}\label{GARCH}
In this section we will apply Proposition \ref{mainprop} to the special case of a GARCH$(1,1)$ process. We suggest a simple way for simulation from the limiting distribution 
\begin{equation}\label{eq:ldx}
\lim_{x\to\infty} \mathcal L\left( \frac{\zeta_{-m}}{x},\ldots, \frac{\zeta_n}{x} \;\middle\vert\;  \zeta_0 > x\right), \;\;\; m,n \in \mathbb{N}_0,
\end{equation} 
of the tail process of $(\zeta_t)_{t \in \mathbb{Z}}$, where $(\zeta_t,\sigma_t)_{t \in \mathbb{Z}}$ are given by \eqref{GARCHX} and \eqref{GARCHsigma}. To this end, we shall make use of Proposition \ref{mainprop} and initially focus on the tail chain of the volatility sequence $(\sigma_t)_{t \in \mathbb{Z}}$ conditioned on the event $\{|\zeta_0|> x\}$. We will then obtain the desired distribution in \eqref{eq:ldx} by the special structure of $(\zeta_t)_{t \in \mathbb{Z}}$. We shall henceforth assume that $\alpha_0>0$ in order to preclude a degenerate solution to~(\ref{GARCHsigma}). Let further $\alpha_1 >0$, $\beta_1 \ge 0, \alpha_1 +
\beta_1 < 1$, and let $\epsilon_t, t \in \mathbb{Z},$ i.i.d.\ standard normal such that there exists a strictly stationary process which satisfies the aforementioned definitions, cf.~\cite{nelson}. Recall that the case $\beta_1 =0$ corresponds to the ARCH(1) model.

It is well known that the marginal distributions of a stationary \linebreak GARCH$(p,q)$ process with standard normal innovations show a regularly varying behavior. In the case of a GARCH$(1,1)$ process there exists a particularly simple characterization of the corresponding index~$\alpha$ of regular variation of the squared processes $(\zeta_t^2)_{t \in \mathbb{Z}}$ and $(\sigma_t^2)_{t \in \mathbb{Z}}$ in terms of the unique positive solution to
\begin{equation}\label{alphaequation}
E\left( \left[\alpha_1 \epsilon_0^2 +\beta_1\right]^{\alpha}\right)=1
\end{equation}
(cf.\ \cite{handbook}, Theorem 1 and Example 2). 

In order to apply Proposition \ref{mainprop} we verify that Conditions~1 to 3 are satisfied for a GARCH$(1,1)$ process. It follows from the aforementioned regular variation of the marginal distribution of $\sigma_0^2$ (and hence also of $|\sigma_0|$) that
$$ \lim_{x \to \infty}\frac{P(|\sigma_0|>ux)}{P(|\sigma_0|>x)}=u^{-2\alpha}, \;\;\; \mbox{ for all } u>0. $$
Since $\epsilon_1$ is independent of $\sigma_0$, and all moments of $|\epsilon_1|$ exist we may apply Breiman's Theorem (cf.\ \cite{breiman}) in order to find that
\begin{equation}\label{Breiman} \lim_{x \to \infty} \frac{P(|\zeta_0|>x)}{P(|\sigma_0|>x)}=\lim_{x \to \infty} \frac{P(|\sigma_0 \epsilon_1|>x)}{P(|\sigma_0|>x)}=E(|\epsilon_1|^{2 \alpha})\end{equation} 
and Condition 1 is satisfied. By the specifications of $\Phi$ and $\Psi$ given in \eqref{GARCHX} and \eqref{GARCHsigma} we get that
\begin{equation}\label{eq:pz}
\lim_{x\to\infty} x^{-1} \Phi(x,e) = \phi(e) = \sqrt{\alpha_1 e^2 +\beta_1}, \;\;\; \lim_{x\to\infty} x^{-1} \Psi(x,e) = \psi(e) = e,
\end{equation}
where both $\Phi$ and $\Psi$ are only defined for $x\geq 0$. This shows that Condition 2 holds where we dropped the second argument in $\phi$ and $\psi$ for simplicity. Concerning Condition 3 there are several instructive arguments: First, it is a direct consequence of Proposition \ref{mrvandcond3} given the multivariate regular variation of the vector $(|\zeta_0|,\sigma_0, \sigma_1)$. Alternatively, it follows from an application of Proposition \ref{onehelpsforall} in connection with Lemma \ref{pointmass} using that $C=E(|\chi|^{2\alpha})$ by \eqref{Breiman}. Finally, the following lemma states both the existence and the specific distribution of $(\sigma_0^{(\zeta )}, \sigma_1^{(\zeta )})$.
\begin{lemma}\label{l:czv}
For the stationary processes $(\zeta_t)_{t \in \mathbb{Z}}$ and $(\sigma_t)_{t \in \mathbb{Z}}$ given by~(\ref{GARCHX})
and~(\ref{GARCHsigma}) there exists a random vector $(\zeta_0^{(\zeta )}, \epsilon_1^{(\zeta )} )$ such that
\begin{equation}\label{eq:czv}
\lim_{x\to\infty} \mathcal L\left( \frac{|\zeta_0|}{x},\epsilon_1,\frac{\sigma_0}{x},
  \frac{\sigma_1}{x}\;\middle\vert \; |\zeta_0|>x\right) = \mathcal L \left(|\zeta_0|^{(\zeta )}, \epsilon_1^{(\zeta )}, \sigma_0^{(\zeta )}, \sigma_1^{(\zeta )}\right),
\end{equation}
$$\mbox{with}\;\;\; \sigma_0^{(\zeta )}=\frac{|\zeta_0|^{(\zeta )}}{|\epsilon_1^{(\zeta )}|}, \;\; \sigma_1^{(\zeta )}=\frac{|\zeta_0|^{(\zeta )}}{|\epsilon_1^{(\zeta )}|}\phi(\epsilon_1^{(\zeta )}),$$ 
where $|\zeta_0|^{(\zeta )}$ and $\epsilon_1^{(\zeta )}$ are independent with $|\zeta_0|^{(\zeta )} \sim \mbox{Par}(2\alpha)$ and $\epsilon_1^{(\zeta )}$ is a symmetric random variable such that $(\epsilon_1^{(\zeta )})^2/2$ is Gamma distributed with shape parameter $\alpha+1/2$ and scale parameter 1. 
\end{lemma}
\begin{proof}
We show that
$$ \lim_{x \to \infty}\mathcal L\left( \frac{|\zeta_0|}{x},\epsilon_1\;\middle\vert \; |\zeta_0|>x\right) = \mathcal L \left(|\zeta_0|^{(\zeta )}, \epsilon_1^{(\zeta )}\right),$$
where $\mathcal L \left(|\zeta_0|^{(\zeta )}, \epsilon_1^{(\zeta )}\right)$ is according to the above proposition. The statement then follows from an application of the continuous mapping theorem in connection with Condition 2. Now, for $s\geq 1, t \in \mathbb{R}$ and $v \in \{-1,1\}$ we have
\begin{eqnarray*}
&& \lim_{x \to \infty}P\left( \frac{|\zeta_0|}{x}>s,|\epsilon_1|\leq t, \sign(\epsilon_1)=v\;\middle\vert \; |\zeta_0|>x\right)\\
&=& \lim_{x \to \infty}\frac{P\left(  \frac{|\zeta_0|}{x}>s,|\epsilon_1|\leq t,\sign(\epsilon_1)=v,|\zeta_0|>x\right)}{P(|\zeta_0|>x)} \\
&=& \frac{s^{-2\alpha}}{E\left(|\epsilon_1|^{2\alpha}\right)}\lim_{x \to \infty}\frac{P\left(|\sigma_0\epsilon_1|>sx,|\epsilon_1|\leq t, \sign(\epsilon_1)=v\right)}{P(\sigma_0>sx)}\\
&=& \frac{s^{-2\alpha}}{E\left(|\epsilon_1|^{2\alpha}\right)}\lim_{x \to \infty}\int_0^t \frac{P(\sigma_0>sx/u)}{P(\sigma_0>sx)} \, F^{|\epsilon_1|}(du)P(\sign(\epsilon_1)=v)\\
&=& \frac{s^{-2\alpha}}{2E\left(|\epsilon_1|^{2\alpha}\right)}\int_0^t u^{2\alpha} \, F^{|\epsilon_1|}(du) \\
&=& s^{-2\alpha}\frac{1}{2} \int_0^{t^2/2}\frac{1}{\Gamma(\alpha+1/2)}y^{\alpha-1/2}e^{-y}\, dy,
\end{eqnarray*}
where the penultimate equality follows by uniform convergence of regularly varying functions, cf.\ \cite{bingham}. The last equality holds by substitution and using that $E(|\epsilon_1|^{2\alpha})=\pi^{-1/2}2^\alpha\Gamma(\alpha+1/2)$. Now, the final expression has product form and hence $|\zeta_0|^{(\zeta )}$, $|\epsilon_1^{(\zeta )}|$ and $\sign(\epsilon_1^{(\zeta )})$ are independent random variables with the stated distributions.
\end{proof}
Having checked that all three conditions are met we may now apply Proposition \ref{mainprop} to the GARCH$(1,1)$ setting.
\begin{proposition}\label{p:js}
Let the stationary time series $(\zeta_t)_{t \in \mathbb{Z}}$ and $(\sigma_t)_{t \in \mathbb{Z}}$ be given by~(\ref{GARCHX})
and~(\ref{GARCHsigma}). Then, for all $m,n\in\mathbb N_0$, as $x\to\infty$,
\begin{equation}
\label{eq:ff}
\mathcal L\left( \frac{\sigma_{-m}}{x},\ldots, \frac{\sigma_n}{x}\;\middle\vert
  \; |\zeta_0| > x\right) \to \mathcal L\left(\sigma_{-m}^{(\zeta )},\ldots,
  \sigma_{n}^{(\zeta )} \right)
\end{equation}
with $(\sigma_0^{(\zeta )},\sigma_1^{(\zeta )})$ as in Lemma~\ref{l:czv}, and 
\begin{equation}
\label{eq:n2}
\sigma_t^{(\zeta )} = \sigma_{t-1}^{(\zeta )} A_t,\quad t\geq 2, \quad\mbox{and}\quad
\sigma_{-t}^{(\zeta )} = \sigma_{-t+1}^{(\zeta )} A_{-t},
\quad t\geq 1, 
\end{equation}
with $A_t, t\geq 2,$ defined by
\begin{align}
A_t& =  \phi(\hat \epsilon_{t}) = \sqrt{\alpha_1 \hat \epsilon_{t}^2 +\beta_1} \label{eq:cs2}
\end{align}
for an i.i.d.\ sequence  $(\hat \epsilon_t)_{t\geq 2}$ independent of $(\sigma_0^{(\zeta )},\sigma_1^{(\zeta )})$ with standard normal distribution, and $A_{-t}, t \in \mathbb{N},$ are i.i.d.\ random variables on $(0, \beta_1^{-1/2})$ (i.e. on $(0, \infty)$ for $\beta_1=0$) independent
of~$(A_t)_{t\geq 2}$ and $(\sigma_0^{(\zeta )},\sigma_1^{(\zeta )})$ with distribution function
\begin{equation}
P(A_{-1}\le x) = \sqrt{2/\pi}\int_{\left(\frac{x^{-2}-\beta_1}{\alpha_1}\right)^{1/2}}^{\infty}
  (\alpha_1 z^2 + \beta_1)^{\alpha} \exp\left(-\frac12 z^2\right) dz\label{eq:cs3}
\end{equation}
for  $0<x<\beta_1^{-1/2}$.
\end{proposition}
\begin{proof}
We apply Proposition \ref{mainprop}. Since $(\sigma_t)_{t \in \mathbb{Z}}\geq 0$ we omit the variables $(B_t)_{t \in \mathbb{Z}}$ from Proposition \ref{mainprop} and restrict ourselves to the distribution of $(A_t)_{t \in \mathbb{Z}}$. Now, \eqref{eq:cs2} follows from Proposition \ref{segers1}. With $\mu=\mathcal{L}(1,A_2)$ we are left to show that the distribution defined in \eqref{eq:cs3} equals the distribution of the second component of the adjoint measure $\mu^\ast=\mathcal{L}(1,A_{-1})$ corresponding to the BFTC($2\alpha, \mu)$. Using \eqref{easyadjointformula} (cf.\ also (4.3) in \cite{segers}) we have that
$$
 P(A_{-1}\le x)= E\left[\mathds{1}_{\{(\alpha_1 \hat{\epsilon}_1^2+\beta_1)^{-1/2}\leq x\}} (\alpha_1
  \hat{\epsilon}_1^2+\beta_1)^{\alpha}\right]+\left(1-E\left[(\alpha_1
  \hat{\epsilon}_1^2+\beta_1)^{\alpha}\right]\right),
$$
for a standard normal $\hat{\epsilon}_1$. This gives \eqref{eq:cs3} since the second summand equals zero by \eqref{alphaequation}. With $A_t, t \geq 2$ or $t \leq 0$, as above the assertion follows with Proposition \ref{mainprop}.
\end{proof}
Next, we take Proposition \ref{p:js} as a starting point for the derivation of \eqref{eq:ldx}.
To this end, note that by~(\ref{GARCHsigma}) we have
\begin{equation}\label{eq:zt}
\mathcal L(\epsilon_{t})  = \mathcal L\left( \left( \frac{\sigma_{t}^2-\alpha_0}{\sigma_{t-1}^2} -
  \beta_1\right)^{1/2} \alpha_1^{-1/2}
S_{t}\right),\quad t\in\mathbb Z,
\end{equation}
for a sequence $(S_t)_{t\in\mathbb Z}$ of i.i.d.\ random variables independent
of~$(\sigma_t)_{t\in\mathbb Z}$ with
\begin{equation}\label{eq:pm}
P(S_0=1) = P(S_0= -1) = 1/2.
\end{equation}
Now, by an application of the continuous mapping theorem in connection with Proposition~\ref{p:js} and \eqref{eq:zt} we get
\begin{eqnarray}\label{eq:mcl}
\nonumber \lim_{x \to \infty} \mathcal L \left( \frac{\zeta_{-m}}{x},\ldots,\frac{\zeta_n}{x}\;\middle \vert\; |\zeta_0|
  > x\right)&=& \mathcal L\left( \sigma_{-m}^{(\zeta )} \epsilon_{-m+1}^{(\zeta )},\ldots,
  \sigma_{n}^{(\zeta )} \epsilon_{n+1}^{(\zeta )} \right)\\ 
&=& \mathcal L\left(
  \zeta_{-m}^{(\zeta )},\ldots,\zeta_n^{(\zeta )}\right),
\end{eqnarray}
with
\begin{align*}
\epsilon_t^{(\zeta )} &= \left(\frac{(\sigma_{t}^{(\zeta )}/\sigma_{t-1}^{(\zeta )})^2 - \beta_1}{\alpha_1}\right)^{1/2}S_t^{(\zeta )}, \;\; t \in \mathbb{Z},
\end{align*}
for a sequence~$(S_t^{(\zeta )})_{t\in\mathbb Z}$ with the same distribution as~$(S_t)_{t\in\mathbb Z}$ and independent of~$(\sigma_t^{(\zeta )})_{t\in\mathbb Z}$.
Note that by the structure of $(\sigma_t^{(\zeta )})_{t \in \mathbb{Z}}$ it follows that
\begin{align}\label{epsilontilde}
\epsilon_t^{(\zeta )} &= \left(\frac{A_t^2 - \beta_1}{\alpha_1}\right)^{1/2}S_t^{(\zeta )}, \,t \geq 2, \;\;\;\mbox{} \;\; \epsilon_t^{(\zeta )} = \left(\frac{A_{t-1}^{-2} - \beta_1}{\alpha_1}\right)^{1/2}S_t^{(\zeta )},\, t\leq 0. 
\end{align}
Now, by~(\ref{eq:cs2}) it holds that $\mathcal L(\epsilon_t^{(\zeta )}) =\mathcal L( \epsilon_t)$, $t \geq 2$.
Further, we have that~$\zeta_0^{(\zeta )}$ is symmetric where
\begin{equation}\label{eq:tx}
P(|\zeta_0^{(\zeta )}| > y) = P(|\zeta_0|^{(\zeta)}>y)=y^{-2\alpha},\quad y\ge 1,
\end{equation} 
 by Lemma \ref{l:czv} and the definition in~(\ref{eq:mcl}).
For simulation from the r.h.s.\ of~(\ref{eq:mcl}) 
it is advantageous to write 
\begin{equation*}
\zeta_{\pm t}^{(\zeta )}=|\zeta_0|^{(\zeta )}\prod_{i=1}^t \frac{|\zeta_{\pm i}^{(\zeta )}|}{|
  \zeta_{\pm (i-1)}^{(\zeta )}|}\sign(\zeta_{\pm t}^{(\zeta )})=|\zeta_0|^{(\zeta )} \prod_{i=1}^t \frac{\sigma_{\pm i}^{(\zeta )}|
  \epsilon_{\pm i+1}^{(\zeta )}|}{\sigma_{\pm (i-1)}^{(\zeta )}|\epsilon_{\pm (i-1)+1}^{(\zeta )}|}\sign(\zeta_{\pm t}^{(\zeta )})
\end{equation*}
for $t\in\mathbb N$, such that replacing for~(\ref{eq:czv}), (\ref{eq:n2}) and~(\ref{epsilontilde}) yields the following proposition.
\begin{proposition}\label{easysimulation} For $m,n \in \mathbb{N}$ let 
\begin{itemize}
 \item $A_t$, $-m \leq t \leq -1,$ distributed according to~(\ref{eq:cs3}),
 \item $S_{t}^{(\zeta )}$, $-m \leq t \leq n,$ i.i.d.\ with the distribution of $S_0$ in~(\ref{eq:pm}),  
 \item $\epsilon_t^{(\zeta )}, 2 \leq t \leq n+1,$ i.i.d.\ standard normal
 \item $\epsilon_1^{(\zeta )}$ according to Lemma \ref{l:czv} and
 \item $|\zeta_0|^{(\zeta )} \sim \mbox{Par}(2\alpha)$
\end{itemize}
be mutually independent random variables. Then,
\begin{eqnarray}\label{eq:dwi}
\nonumber&& \lim_{x\to\infty} \mathcal L\left( \frac{\zeta_{-m}}{x},\ldots, \frac{\zeta_0}{x}, \ldots, \frac{\zeta_n}{x} \;\middle\vert\;  |\zeta_0| > x\right) \\
\nonumber &=& \mathcal L\Bigg( |\zeta_0|^{(\zeta )}\bigg(\prod_{i=1}^m A_{-i}\bigg)\frac{\left((A_{-m}^{-2}-\beta_1)/\alpha_1\right)^{1/2}}{|\epsilon_1^{(\zeta )}|}S_{-m}^{(\zeta )},\ldots,|\zeta_0|^{(\zeta )} S_0^{(\zeta )},\ldots,\\
 &&
 \quad 
        |\zeta_0|^{(\zeta )}
          \bigg(\prod_{i=1}^n(\alpha_1 (\epsilon_{i}^{(\zeta )})^2 +\beta_1)^{1/2}\bigg)\frac{|\epsilon_{n+1}^{(\zeta )}|}{|\epsilon_1^{(\zeta )}|}S_n^{(\zeta )}\Bigg).
\end{eqnarray}

\end{proposition}
Finally, conditioning on~$\zeta_0>x$ as in~(\ref{eq:ldx}) instead of~$|\zeta_0|>x$ leads
to the same limit distribution as in~(\ref{eq:dwi}) but with $S_0^{(\zeta )}=1$ almost surely.

\section{Numerical Example}\label{simulation}

Our analysis of two connected time series was originally motivated by an idea to extend the approach considered in~\cite{deHaan} for ARCH$(1)$ processes
 to a simulation study for extremal characteristics of the more general GARCH$(1,1)$ class. Among such extremal measures 
 the extremal index is a well-known example. It characterizes the behavior of extreme events in a time series, 
i.e.\ the strength of dependence between subsequent high-level exceedances. More precisely, let $M_n:=\max (X_1, \dots, X_n)$ for $n \in \mathbb{N}$
 where~$(X_t)_{t \in \mathbb{N}}$ is a stationary univariate process with marginal distribution function~$F$. Let further~$(\widetilde{X_t})_{t \in \mathbb{N}}$ be the associated i.i.d.\ sequence with the same marginal distribution~$F$, and let accordingly~$\widetilde{M_n}:=\max (\widetilde{X_1}, \dots, \widetilde{X_n})$ for $n \in \mathbb{N}$. Assume that there exists a nonnegative number~$\theta_X$ such that for every~$\tau>0$ there is a sequence $(u_n)_{n \in \mathbb{N}}$ such that
 $$ \lim_{n \to \infty} P(\widetilde{M_n} \leq u_n)=e^{-\tau} \;\;\; \mbox{and} \;\;\; \lim_{n \to \infty} P(M_n \leq u_n) = e^{-\theta_X \cdot \tau}. $$
Then,~$\theta_X$ is called the extremal index of the process~$(X_t)_{t \in \mathbb{N}}$, and~$\theta_X \in [0,1]$. Note also that under an additional mild mixing condition the extremal index corresponds to the inverse of the mean cluster size of extreme values in the series, cf.~\cite{embrechts} for reference and for further details about the extremal index.  
Now, focussing on the GARCH$(1,1)$ model as defined in Section~\ref{s:intro} we find that~$\theta_{\zeta}=\lim_{m \to \infty}\theta_{\zeta,m}$ where
 \begin{equation}\label{thetam}\theta_{\zeta,m}=\lim_{x \to \infty}P\left(\max(\zeta_1, \dots, \zeta_m)<x \mid \zeta_0>x\right), \;\;\; m \in \mathbb{N}, \end{equation}
cf.\ \cite{embrechts} and \cite[Section 5.2]{ehlert}. 

In addition to the extremal index we shall in the following also consider two alternative extremal characteristics that may be evaluated by the same simulation approach. The so-called extremal coefficient function discussed in~\cite{fasen} is given by
 \begin{equation}\label{chi}\chi_\zeta(h)=\lim_{x \to \infty} P(\zeta_h>x \mid  \zeta_0>x) \end{equation}
for $h \in \mathbb{Z}$. 
Following the notion of usual autocovariances the extremal coefficient function gives the conditional probability of two extreme events separated by a lag~$h\in\mathbb Z$. 
For two reasons we will also briefly describe a modification of this concept,
i.e.\  a probability for threshold exceedances at a lag~$h\in\mathbb N$ given that $\zeta_0$ is 
not only extreme as in~\eqref{chi}, but given that~$\zeta_0$ is also at the
beginning of an extremal cluster in a time series, cf.~\cite[Chapter~5]{ehlert} for a discussion. 
More precisely, let
$\gamma_\zeta(h)=\lim_{m \to \infty}\gamma_{\zeta,m}(h), h \in \mathbb{N},$ where
 \begin{equation}\label{gammam} \gamma_{\zeta,m}(h)=\lim_{x \to \infty} P(\zeta_h>x \mid \zeta_0>x \mbox{ and } \zeta_i\leq x, i=-m, \dots, -1).\end{equation}
The first reason to touch on this characteristic in our study is its potential to serve as a complement to
the extremal coefficient function regarding questions of cluster structures in risk management and related applications that focus on the development of extremal events.
The second reason is related to the numerical simulation of~(\ref{thetam}) to~(\ref{gammam})
that will be based on the tail chain concept discussed in Section~\ref{GARCH}. While the evaluation of~$\theta$ and the extremal coefficient function~$\chi$ requires a series of runs of either the forward or the backward tail chain it is evident from~(\ref{gammam}) that the simulation of~$\gamma$ must be based on simultaneous runs of the forward and the backward tail chain at the same time.
Note at this point that we are not aware of any general closed form solutions for~(\ref{thetam}) to~(\ref{gammam}) that would include the GARCH$(p,q)$ model parameters, not even for $p=q=1$.

Our simulation setup generalizes similar methods proposed by~\cite{deHaan} and~\cite{laurini}. The algorithm used in~\cite{deHaan} is restricted to single time series which satisfy the assumptions of \cite{segers}
that do, however, not hold for the log returns in a GARCH$(1,1)$ setting. As a generalization of~\cite{laurini} our algorithm is in principle not restricted to models with symmetric innovations. Furthermore, to our knowledge, there have been no approaches to simulate from the backward tail chain so far. 

Taking the limit $x \to \infty$ in~\eqref{thetam} to~\eqref{gammam} we note that as indicated above all three characteristics can be expressed via the tail chain distribution. By simulation from this distribution (cf.\ Proposition \ref{easysimulation}) we may therefore evaluate these quantities by Monte Carlo estimation. In Table~\ref{t:es} we report the results of such a simulation study for $\theta_{\zeta,m}$ and $\gamma_{\zeta,m}(h), m=500$, (which we use as approximations
 for $\theta_\zeta$ and $\gamma_\zeta(h)$) and $\chi_\zeta(h)$ for $h=1,2,3$. The evaluation of probabilities is based on $N=10000$ draws. We fix $\alpha_1 + \beta_1= 0.99$ in the table
in order to reflect the stylized fact that $\alpha_1 + \beta_1$ is close to~one in many applications. The last row of Table \ref{t:es} is motivated by the following example.

\renewcommand{\tabcolsep}{4pt}
\begin{table}[h!]
\centering
\begin{tabular}{cccccccccc}
\hline $\alpha_1$ & $\beta_1$ & $\alpha $    & $\hat\theta_{\zeta,m}$ &
$\hat\chi_\zeta(1)$
&$\hat\chi_\zeta(2)$ & $\hat\chi_\zeta(3)$ &$\hat\gamma_{\zeta,m}(1)$ & $\hat\gamma_{\zeta,m}(2) $ & $\hat\gamma_{\zeta,m}(3) $
\rule[5.3ex]{0cm}{0cm} \rule[-3.3ex]{0cm}{0cm} 
\\\hline
0.99 & 0  & 1.014  & 0.570   &0.213 &
0.139 &0.104 &0.251  &0.167 & 0.125     
\\  
0.15 & 0.84  & 1.478  &   0.207  &  0.061  &0.063
 &0.065 & 0.153 & 0.144 & 0.139       
\\  
0.11 & 0.88  & 1.838  &  0.245  &   0.052&
0.042 & 0.038 & 0.110& 0.104 & 0.104
\\
0.09 & 0.90  & 2.203  & 0.304 &0.045 & 0.035
 &0.034 &0.089  & 0.085 & 0.081     
    \\  
0.07 & 0.92  & 2.885  & 0.397  &0.022 &0.020
 &0.020 & 0.055 & 0.050 & 0.053     
    \\  
0.04 & 0.95  & 5.991
  & 0.854  &  0.005 &0.004 &
0.003 &0.007 & 0.007 & 0.006 
  \\  
0.072 & 0.920 &   2.476  &  0.317    &  0.021&
0.020 & 0.027 & 0.063 & 0.064 & 0.066     \\\hline
\end{tabular}
\caption{Extremal measures ($m=500$) for selected GARCH(1,1)
 processes with $\alpha_1 + \beta_1 = 0.99$, as well as the process fitted in Example~\ref{ex:garch}. The results are based on $N=10000$ runs of the tail chain. The
 approximate confidence intervals are smaller than $\pm  0.01$ for all entries.}\label{t:es}
\end{table}
\begin{example}\label{ex:garch}
We fit the GARCH(1,1) model given by~(\ref{GARCHX}) to a data set of log returns of the S\&P~500 index from 01.04.80 to 30.03.10 (7569 records).
The estimated parameters \cite{traplettihornik09} are 
\begin{equation}\label{eq:pgs+p}       
\hat \alpha_0 =0.1\times 10^{-5}\; (10^{-7}),\quad           \hat \alpha_1 = 0.072\; (0.002),\quad          \hat \beta_1 = 0.920\;(0.003)
\end{equation}
where the ML standard errors are given in brackets. 
We include an evaluation of the corresponding extremal
measures by the above tail chain approach in the last row of Table~\ref{t:es}. In order to discuss the adequacy of a GARCH$(1,1)$ model with regard to the extremal behavior we compare the result of Table \ref{t:es} with the so-called blocks estimator  of the extremal index for the given data~\cite{wuertz,BGST}. For a block length of $m=126$ and a threshold corresponding to the empirical 0.95~quantile the estimator yields $\hat{\theta}=0.305\, (0.210, 0.515)$. 
Here, the brackets 
represent the simulated~95\%
confidence interval which is based on $N=1000$~independent GARCH(1,1) processes of length 7569 according to~\eqref{eq:pgs+p}.
As to the choice of the block length note that extremal events occuring in two distinct blocks are assumed to be independent. Here, six trading months correspond to 126 days and appear to be a reasonable order of magnitude. 
Given that our block length is a valid choice the fact that the result
falls within the simulated confidence interval indicates a
satisfactory agreement of the data set and a GARCH$(1,1)$ model with regard to their extremal behavior.
\end{example}

\section*{Appendix}
\paragraph{Details on the construction of $f$ in the proof of Lemma \ref{nonuniquenesslemma}}\quad \newline
Let $z = z_{i}=5^i, i \in \mathbb{N}_0,$ and $I_{1} = [z,2\frac{1}{4} z], I_{2} = [2\frac{1}{4} z, 3z],I_{3}=[3z,4z],I_{4}=[4z,5z]$. The restrictions $f_{i}:= f|_{I_{i}}$ will look as follows: $f_{1}$ is strictly increasing from 
value $z$ to $3z$, $f_{2}$ is symmetric on $I_{2}$, on the left half of $I_{2}$ it increases from $3z$ to $5z$. Finally, $f_{3}$ interpolates linearly between the values $3z$ and $z$, and $f_{4}$ between $z$ and $5z$.\\
Given the definitions of $f_{3}$ and $f_{4}$, we show that $f_{1}$, $f_{2}$ can be defined implicitely.\\
For the definition of  $f_{1}$ via $f_{1}^{-1}$ consider $x \in [z,3z]$. The function $f_{1}^{-1}$ has to satisfy 
$\frac{1}{x} \overset{!}{=} P(f(Y) > x) = P(f_{1}^{-1}(x) < Y < f_{3}^{-1}(x)) + P(f_{4}^{-1}(x) < Y) = \frac{1}{f_{1}^{-1}(x)} - \frac{1}{f_{3}^{-1}(x)} + \frac{1}{f_{4}^{-1}(x)}$
thus $\frac{1}{f_{1}^{-1}(x)} \overset{!}{=} \frac{1}{x} + \frac{1}{f_{3}^{-1}(x)} - \frac{1}{f_{4}^{-1}(x)} =: h(x)$.
Observe that $h'(x) = -\frac{1}{x^2} + \frac{1/2}{(f_{3}^{-1}(x))^2} + \frac{1/4}{(f_{4}^{-1}(x))^2} \leq -\frac{1}{(3z)^2} + \frac{1/2}{(3z)^2} + \frac{1/4}{(4z)^2} < 0$. Thus $f_{1}^{-1}(x)$ is strictly increasing and $f_{1}^{-1}(z) = z$, $f_{1}^{-1}(3z) = 2\frac{1}{4}z$ by the above formula. This shows that $f_{1}$ is well-defined as the inverse of $f_{1}^{-1}$.\\
For the  existence of $f_{2}$ on $I_{2}$ as described it suffices to show that $h(x) := \frac{1}{x} - P(f_{4}^{-1}(x) \geq Y) = \frac{1}{x} - \frac{1}{f_{4}^{-1}(x)}, x \in [3z,5z]$, is strictly decreasing (note that $h(5z) = 0$) which follows from $h'(x) = -\frac{1}{x^2} + \frac{1/4}{(f_{4}^{-1}(x))^2} \leq -\frac{1}{(5z)^2} + \frac{1/4}{(4.5z)^2} <0$.




\end{document}